\documentclass[11pt,a4paper]{article}
\date{}

\usepackage{amssymb}
\usepackage{latexsym}
\usepackage{amsmath,amssymb}
\usepackage{amsthm,enumerate,verbatim}
\usepackage{amsfonts}
\usepackage{graphicx}
\usepackage{xcolor}
\usepackage{subfigure}
\usepackage{multimedia}
\newcommand{\N}{\noindent}

\newcommand{\R}{{\mathbb{R}}}

\newcommand{\BEE}{\begin{equation*}}
\newcommand{\EEE}{\end{equation*}}
\newcommand{\BE}{\begin{equation}}
\newcommand{\EE}{\end{equation}}

\numberwithin{equation}{section}

\newtheorem{definition}{Definition}[section]
\newtheorem{theorem}{Theorem}[section]
\newtheorem{lemma}{Lemma}[section]
\newtheorem{example}{Example}[section]
\newtheorem{proposition}{Proposition}[section]
\newtheorem{remark}{Remark} [section]
\newtheorem{corollary}{Corollary}[section]

\def\RR{\hbox{I\kern-.2em\hbox{R}}}
\title{\bf\large Group sparse optimization via $\ell_{p,q}$ regularization}
\author{Yaohua Hu\thanks{College of Mathematics and Statistics, Shenzhen University, Shenzhen 518060, P. R. China (hyh19840428@163.com).},
\quad Chong Li\thanks{Department of Mathematics, Zhejiang University, Hangzhou 310027, P. R. China (cli@zju.edu.cn).},
\quad Kaiwen Meng\thanks{School of Economics and Management, Southwest Jiaotong University, Chengdu 610031, P. R. China (mkwfly@126.com).},
\quad Jing Qin\thanks{School of Life Sciences, The Chinese University of Hong Kong, Shatin, New Territories, Hong Kong (qinjing@cuhk.edu.hk).},
\quad Xiaoqi Yang\thanks{Department of Applied Mathematics, The Hong Kong Polytechnic University, Kowloon, Hong Kong (mayangxq@polyu.edu.hk).}
}
\begin{document}

\maketitle

\noindent {\bf Abstract}\quad
In this paper, we investigate a group sparse optimization problem via $\ell_{p,q}$ regularization in three aspects: theory, algorithm and application. In the theoretical aspect, by introducing a notion of group restricted eigenvalue condition, we establish some oracle property and a global recovery bound of order $O(\lambda^\frac{2}{2-q})$ for any point in a level set of the $\ell_{p,q}$ regularization problem, and by  virtue of modern variational analysis techniques, we also provide a local analysis of recovery bound of order $O(\lambda^2)$ for a path of local minima.
In the algorithmic aspect, we apply the well-known proximal gradient method to solve the $\ell_{p,q}$ regularization problems, either by analytically solving some specific $\ell_{p,q}$ regularization subproblems,  or by using the Newton method to solve general $\ell_{p,q}$ regularization subproblems.
In particular, we establish the linear convergence rate of the proximal gradient method for solving the $\ell_{1,q}$ regularization problem under some mild conditions. As a consequence, the linear convergence rate of proximal gradient method for solving the usual $\ell_{q}$ regularization problem ($0<q<1$) is obtained. Finally in the aspect of application, we present some numerical results on both the simulated data and the real data in gene transcriptional regulation.\\


\noindent {\bf Key words}\quad Group sparse optimization, $\ell_{p,q}$ regularization, nonconvex optimization, restricted eigenvalue condition, proximal gradient method, iterative thresholding algorithm, gene regulation network.

\section{Introduction}
In recent years, a great amount of attention has been paid to the sparse optimization problem, which is to find the sparse solutions of
an underdetermined linear system. The sparse optimization problem arises in a wide range of fields, such as variable selection, pattern analysis, graphical modeling and compressive sensing; see \cite{Blumensath08,Candes06b,ChenSS01,Donoho06,Fan01,Tibshirani94} and references therein.

In many applications, the underlying data usually can be represented approximately by a linear system of the form
\[
Ax=b + \varepsilon,
\]
where $A\in \R^{m\times n}$ and $b\in \R^m$ are known, $\varepsilon\in \R^m$ is an unknown noise vector,
and $x=(x_1,x_2,\dots,x_n)^\top \in \R^n$ is the variable to be estimated. If $m\ll n$, the above linear system is seriously ill-conditioned and may have infinitely many solutions.
The sparse optimization problem is to recover $x$ from information $b$ such that $x$ is of a sparse structure. The sparsity of variable $x$ has been measured by the $\ell_p$ norm $\|x\|_p$ ($p=0$, see \cite{Blumensath08,BrediesLorenz08}; $p=1$, see \cite{BeckTeboulle09,ChenSS01,Daubechies04,Donoho06,Tibshirani94,StWright2009,YZ11}; and $p=1/2$, see \cite{Chartrand08,XuZB12}).
The $\ell_p$ norm $\|x\|_p$ for $p>0$ is defined as
\[
\|x\|_p:=\left(\sum_{i=1}^n |x_i|^p\right)^{1/p},
\]
while the $\ell_0$ norm $\|x\|_0$ is defined as the number of nonzero components of $x$.
The sparse optimization problem can be modeled as
\begin{equation*}
  \begin{array}{ll}
     \text{min}&\|Ax-b\|_2\\
     \text{s.t.}  &\|x\|_0\le s,
  \end{array}
\end{equation*}
where $s$ is the given sparsity level.

For the sparse optimization problem, a popular and practical technique is the regularization method, which is to transform the sparse optimization problem into
an unconstrained optimization problem, called the regularization problem. For example, the $\ell_0$ regularization problem is
\begin{equation*}\label{eq-L0}
\min_{x\in \R^n} \|Ax-b\|_2^2+\lambda \|x\|_0,
\end{equation*}
where $\lambda>0$ is the regularization parameter, providing a tradeoff between accuracy and sparsity.
However, the $\ell_0$ regularization problem is nonconvex and non-Lipschitz, and thus it is generally intractable to solve it directly
(indeed, it is NP-hard; see \cite{Natarajan95}).

To overcome this difficulty, two typical relaxations of the $\ell_0$ regularization problem are introduced, which are the $\ell_1$ regularization problem
\begin{equation}\label{eq-L1}
\min_{x\in \R^n} \|Ax-b\|_2^2+\lambda \|x\|_1
\end{equation}
and the $\ell_{1/2}$ regularization problem
\begin{equation}\label{eq-L1/2}
\min_{x\in \R^n} \|Ax-b\|_2^2+\lambda \|x\|_{1/2}^{1/2}.
\end{equation}

\subsection{$\ell_p$ regularization problems}
The $\ell_1$ regularization problem, also called Lasso \cite{Tibshirani94} or Basis Pursuit \cite{ChenSS01}, has attracted much attention and has been accepted as one of the most useful tools for the sparse optimization problem.
Since the $\ell_1$ regularization problem is a convex optimization problem, many exclusive and efficient algorithms have been proposed and developed for solving \eqref{eq-L1}, for instance, the interior-point methods \cite{Candes06b,ChenSS01}, LARs \cite{Efron04}, the gradient projection method \cite{Figueiredo07} and the alternating direction method \cite{YZ11}.
However, in many practical applications, the solutions obtained from the $\ell_1$ regularization problem are much less sparse than those of the $\ell_0$ regularization problem, and it often leads to sub-optimal sparsity in reality; see, e.g., \cite{Chartrand08,XuZB12,Zhang10}.

Recently, the $\ell_{1/2}$ regularization problem is proposed to improve the performance of sparsity recovery of the $\ell_1$ regularization problem. Extensive computational studies in \cite{Chartrand08,XuZB12} revealed that the $\ell_{1/2}$ regularization problem admits a significantly stronger sparsity promoting property than the $\ell_1$ regularization problem in the sense that it guarantees to achieve the sparse solution from a smaller amount of samples. However, the $\ell_{1/2}$ regularization problem is nonconvex, nonsmooth and non-Lipschitz, and thus it is very difficult in general to design efficient algorithms for solving it. It was presented in \cite{Ge10} that finding the global minimal value of the $\ell_{1/2}$ regularization problem \eqref{eq-L1/2} is strongly NP-hard, while fortunately, computing a local minimum could be done in polynomial time. Some fast and efficient algorithms have been proposed to find a local minimum of \eqref{eq-L1/2}, such as the hybrid OMP-SG algorithm \cite{ChenXJ10}
and the interior-point potential reduction algorithm \cite{Ge10}.

The $\ell_{1/2}$ regularization problem \eqref{eq-L1/2} is a variant of lower-order penalty problems, investigated in \cite{HuangYang03,LuoPangRalph96,YangHuang01SIAM}, for a constrained optimization problem. The main advantage of the lower-order penalty functions over the classical $\ell_1$ penalty functions is that they require weaker conditions to guarantee an exact penalization property and that their least exact penalty parameter is smaller; see \cite{HuangYang03}.
It was reported in \cite{YangHuang01SIAM} that the first- and second-order necessary optimality conditions of lower order penalty problems converge to that of the original constrained optimization problem under a linearly independent constraint qualification.

Besides the preceding numerical algorithms,
one of the most widely studied methods for solving the sparse optimization problem is the class of the iterative thresholding algorithms, which is studied in a uniform framework of proximal gradient methods; see \cite{BeckTeboulle09,Blumensath08,BrediesLorenz08,Combettes05,Daubechies04,Nesterov13,XuZB12} and references therein.
It is convergent and of very low computational complexity. Benefitting from its simple formulation and low storage requirement, it is very efficient and applicable for large-scale sparse optimization problem.
In particular, the iterative hard (resp. soft, half) thresholding algorithm for the $\ell_0$ (resp. $\ell_1$, $\ell_{1/2}$) regularization problem  was studied in \cite{Blumensath08,BrediesLorenz08} (resp. \cite{BeckTeboulle09,Daubechies04}, \cite{XuZB12}).

\subsection{Global recovery bound}


To estimate how far is the solution of regularization problems from that of the linear system, the global recovery bound or $\ell_2$ consistency of the $\ell_1$ regularization problem have been investigated in the literature \cite{Bickel09,Bunea07,Meinshausen09,Geer09,TZhang09}. More specifically, under some mild conditions on $A$, such as the restricted isometry property (RIP) \cite{CandesTao05} or restricted eigenvalue condition (REC) \cite{Bickel09}, van de Geer and B\"{u}hlmann \cite{Geer09} established a deterministic recovery bound for the (convex) $\ell_1$ regularization problem:
\begin{equation}\label{eq-RB-L1}
\|x^*(\ell_1)-\bar{x}\|_2^2=O\left(\lambda^2s\right),
\end{equation}
where $x^*(\ell_1)$ is the solution of \eqref{eq-L1}, $\bar{x}$ is the solution of linear system $Ax=b$, and $s:=\|\bar{x}\|_0$ is the sparsity of $\bar{x}$.
In the statistic literature,  \cite{Bickel09,Bunea07,Meinshausen09,TZhang09} provided the recovery bound in a high probability for the $\ell_1$ regularization problem when the size of the variable tends to infinity, under REC/RIP or some relevant conditions.
However, to the best of our knowledge, the recovery bound for the general (nonconvex) $\ell_p$ regularization problem is still undiscovered. We will establish such a deterministic property in Section 2.

\subsection{Group sparse optimization}
In applications, a wide class of problems usually has certain special structures,
and recently, enhancing the recoverability due to the special structures has become an active topic in the sparse optimization.
One of the most popular structures is the group sparsity structure, that is, the solution has a natural grouping of its components,
and the components within each group are likely to be either all zeros or all nonzeros. In general, the grouping information can be any arbitrary partition of $x$, and it is usually pre-defined based on prior knowledge of specific problems. Let $x:=(x_{\mathcal{G}_1}^\top,\cdots,x_{\mathcal{G}_r}^\top)^\top$ represent the group structure of $x$. The group sparsity of $x$ with such a group structure can be measured by an $\ell_{p,q}$ norm defined by
\[
\|x\|_{p,q}:=\left(\sum_{i=1}^r \|x_{\mathcal{G}_i}\|_p^q\right)^{1/q}.
\]

Exploiting the group sparsity structure can reduce the degrees of freedom in the solution, thereby leading to better recovery performance.
Benefitting from these advantages, the group sparse optimization model has been applied in birthweight prediction \cite{Bach08,Yuan04}, dynamic MRI \cite{Usman11} and gene finding \cite{Meier08,HYang10} with the $\ell_{2,1}$ norm.
More specifically, the following $\ell_{2,1}$ regularization problem
\begin{equation*}\label{eq-L2,1}
\min_{x\in \R^n} \|Ax-b\|_2^2+\lambda \|x\|_{2,1}
\end{equation*}
was introduced by Yuan and Lin \cite{Yuan04} to study the grouped variable selection in statistics under the name of group Lasso.
The $\ell_{2,1}$ regularization, an important extension of the $\ell_1$ regularization, proposes an $\ell_2$ regularization for each group
and ultimately yields the sparsity in the group manner.
Since the $\ell_{2,1}$ regularization problem is a convex optimization problem, some effective algorithms have been proposed, such as, the spectral projected gradient method \cite{Berg2008}, SpaRSA \cite{StWright2009} and the alternating direction method \cite{YZhang11}.

\subsection{The aim of the paper}
In this paper, we will investigate the group sparse optimization via $\ell_{p,q}$ ($p\ge 1,0\le q\le 1$) regularization,
also called the $\ell_{p,q}$ regularization problem
\begin{equation}\label{eq-GSOP}
\min_{x\in \R^n} F(x):=\|Ax-b\|_2^2+\lambda \|x\|_{p,q}^q.
\end{equation}
We will investigate the oracle property and recovery bound for the $\ell_{p,q}$ regularization problem, which extend the existing results in two ways: one is the lower-order regularization problem, including the $\ell_{q}$ $(q<1)$ regularization problem; the other is the group sparse optimization problem, including the $\ell_{2,1}$ regularization problem (group Lasso) as a special case.
To this end, we introduce the weaker notions of REC: the lower-order REC and the group REC (GREC).
We will further establish the relationships between the new notions with the classical one: the lower-order REC is weaker than the classical REC, but the reverse is not true (see Example \ref{ex-REC}); and the GREC is weaker than the REC.
Under the GREC, we will provide the oracle property and the global recovery bound for the $\ell_{p,q}$ regularization problem (see Theorem \ref{thm-RecoverBound}).
Furthermore, we will conduct a local analysis of recovery bound for the $\ell_{p,q}$ regularization problem by virtue of modern variational analysis techniques \cite{Roc98}.
More precisely, we assume that the columns of $A$ corresponding to the active components of $\bar{x}$ (a solution of $Ax=b$) are linearly independent. This leads us to the application of implicit function theorem and thus guarantees the existence of a local path around $\bar {x}$ which satisfies a second-order growth condition. As such, in the local recovery bound, we establish a uniform quadratic recovery bound for all $\ell_{p,q}$ regularization problem, see Theorem 2.2.

The proximal gradient method is one of the most popular and practical methods for the sparse optimization problems, either convex or nonconvex problems.
We will apply the proximal gradient method to solve the $\ell_{p,q}$ regularization problem \eqref{eq-GSOP}.
The advantage of the use of the proximal gradient method is that for some specific regularization problems, the proximal subproblems have the analytical solutions, and the resulting algorithm is thus practically attractive.
In the general cases when the analytical solutions of the proximal optimization subproblems seem not available, we will employ the Newton method to solve the proximal optimization subproblems.
Furthermore, we will investigate the linear convergence rate of proximal gradient method for solving the $\ell_{p,q}$ regularization problem when $p=1$ and $0<q<1$ under the assumption that any nonzero group of a local minimum is active. The problem \eqref{eq-GSOP} of the case $p=1$ and $0<q<1$ possesses  the properties that the regularization term $\|x\|_{p,q}^q$ is concave and the objective function $F(x)$ of \eqref{eq-GSOP} satisfies a second-order growth condition, which play an important role in the establishment of the linear convergence rate.
To the best of our knowledge, this is the first attempt to study the linear convergence rate of proximal gradient method for solving nonconvex optimization problems. As a consequence of this result, we will obtain the linear convergence rate of proximal gradient method for solving $\ell_q$ regularization problem ($0<q<1$), which includes the iterative half thresholding algorithm ($q=1/2$) proposed in \cite{XuZB12} as a special case.
The result on linear convergence rate of proximal gradient method for solving $\ell_q$ regularization problem is still new, as far as we know.

In the aspect of application, we will conduct some numerical experiments on both simulated data and real data in gene transcriptional regulation to demonstrate the performance of the proposed proximal gradient method. From the numerical results, it is demonstrated that the $\ell_{p,1/2}$ regularization is the best one among the $\ell_{p,q}$ regularizations for $q\in[0,1]$, and it outperforms the $\ell_{p,1}$ and $\ell_{p,0}$ regularizations on both accuracy and robustness. This observation is consistent with several previous numerical studies \cite{Chartrand08,XuZB12} on the $\ell_{p}$ regularization problem.

\subsection{Main contributions}
The main objectives of this paper are to establish the oracle property and recovery bound, to design an efficient numerical method for the $\ell_{p,q}$ regularization problem \eqref{eq-GSOP}, and to apply the proposed method to the gene transcriptional regulation. The main contributions are presented as follows.
\begin{enumerate}[(i)]
  \item We establish the following global recovery bound for the $\ell_{p,q}$ regularization problem \eqref{eq-GSOP} under the $(p,q)$-GREC:
    \begin{equation}\label{eq-MC1}
    \|x^*-\bar{x}\|_2^2\le\left\{\begin{matrix}
       O\left(\lambda^\frac{2}{2-q}S\right),&{2^{K-1}q=1,}\\
       O\left(\lambda^\frac{2}{2-q}S^\frac{3-q}{2-q}\right), &{2^{K-1}q>1,}
    \end{matrix}\right.
    \end{equation}
    where $\bar{x}$ is a true solution of $Ax=b$, $S:=\|\bar{x}\|_{p,0}$ is the group sparsity of $\bar{x}$, $0<q\le 1\le p\le 2$, $x^*$ is any point in the level set ${\rm lev}_F(\bar{x})$ of \eqref{eq-GSOP}, and $K$ is the smallest integer such that $2^{K-1}q\ge 1$.
  \item 
  By virtue of the variational analysis technique, for all the $\ell_{p,q}$ regularization problems, we establish a uniform local recovery bound
    \[
    \|x^*_{p,q}(\lambda)-\bar{x}\|_2^2\le O\left(\lambda^2S\right)\quad {\rm for~small}~ \lambda,
    \]
    where $0 < q < 1 \leq p$ and $x^*_{p,q}(\lambda)$ is a local optimal solution of \eqref{eq-GSOP} (near $\bar{x}$).
  \item We present the analytical formulae for the proximal optimization subproblems of some specific $\ell_{p,q}$ regularizations,
        when $p=1,2$ and $q=0,1/2,2/3,1$.
      Moreover, we prove that any sequence $\{x^k\}$, generated by proximal gradient method for solving the $\ell_{1,q}$ regularization problem, linearly converges to a local minimum $x^*$ under some mild conditions, i.e., there exist $N\in \mathbb{N}$, $C>0$ and $\eta\in (0,1)$ such that
\[
F(x^k)-F(x^*)\le C\eta^k\quad {\rm and} \quad \|x^k-x^*\|_2\le C\eta^k,\quad \mbox{for any } k\ge N.
\]

  \item
      Our numerical experiments show that, measured by the biological golden standards, the accuracy of the gene regulation networks forecasting
      can be improved by exploiting the group structure of TF complexes.
      The successful application of group sparse optimization to gene transcriptional regulation will facilitate biologists to study the gene regulation of higher model organisms in a genome-wide scale.
\end{enumerate}

\subsection{The organization of the paper}
This paper is organized as follows.
In section 2, we introduce the concepts of $q$-REC and GREC, and establish the oracle property and (global and local) recovery bounds for the $\ell_{p,q}$ regularization problem.
In section 3, we apply the proximal gradient method to solve the group sparse optimization using different types of $\ell_{p,q}$ regularization,
and investigate the linear convergence rate of the resulting proximal gradient method.
Finally, section 4 exhibits the numerical results on both simulated data and real data in gene transcriptional regulation.

\section{Global and local recovery bounds}
This section is devoted to the study of the oracle property and (global and local) recovery bounds for the $\ell_{p,q}$ regularization problem \eqref{eq-GSOP}. To this end, we first present some basic inequalities of $\ell_p$ norm and introduce the notions of RECs, as well as their relationships.

The notation adopted in this paper is described as follows.
We let the lowercase letters $x,y,z$ denote the vectors, capital letters $N,S$ denote the numbers of groups in the index sets,
caligraphic letters $\mathcal{I}$, $\mathcal{T}$, $\mathcal{S}$, $\mathcal{J}$, $\mathcal{N}$ denote the index sets.
In particular, we use $\mathcal{G}_i$ to denote the index set corresponding to the $i$-th group and
$\mathcal{G}_\mathcal{S}$ to denote the index set $\{\mathcal{G}_i:i\in\mathcal{S}\}$.
For $x\in \R^n$ and $\mathcal{T}\subseteq \{1,\dots,n\}$, we use $x_{\mathcal{T}}$ to denote the subvector of $x$ corresponding to $\mathcal{T}$.

Throughout this paper, we assume that the group sparse optimization problem is of the group structure described as follows.
Let $x:=(x_{\mathcal{G}_1}^\top,\cdots,x_{\mathcal{G}_r}^\top)^\top$ represent the group structure of $x$, where $\{x_{\mathcal{G}_i}\in \R^{n_i}: i=1,\cdots,r\}$ is the grouping of $x$, $\sum_{i=1}^r n_i=n$ and $n_{\max}:=\max\left\{n_i:i\in \{1,\dots,r\}\right\}$.
For a group $x_{\mathcal{G}_i}$, we use $x_{\mathcal{G}_i}=0$ (reps. $x_{\mathcal{G}_i}\neq 0$, $x_{\mathcal{G}_i}\neq_{\mathbf{a}} 0$) to denote a zero (reps. nonzero, active) group, where $x_{\mathcal{G}_i}=0$ means that $x_j=0$ for all $j\in \mathcal{G}_i$; $x_{\mathcal{G}_i}\neq 0$ means that $x_j\neq 0$ for some $j\in \mathcal{G}_i$; and $x_{\mathcal{G}_i}\neq_\mathbf{a} 0$ means that $x_j\neq 0$ for all $j\in \mathcal{G}_i$. It is trivial to see that
\[
x_{\mathcal{G}_i}\neq_\mathbf{a} 0 \quad \Rightarrow \quad  x_{\mathcal{G}_i}\neq 0.
\]

For this group structure and $p>0$, the $\ell_{p,q}$ norm of $x$ is defined by
\begin{equation}\label{eq-Lpq norm}
\|x\|_{p,q}=\left\{\begin{matrix}
   \left(\sum_{i=1}^r \|x_{\mathcal{G}_i}\|_p^q\right)^{1/q},&{q> 0,}\\
   \sum_{i=1}^r \|x_{\mathcal{G}_i}\|_p^0, &q=0,
\end{matrix}\right.
\end{equation}
which proposes the $\ell_p$ norm for each group and then processes the $\ell_q$ norm for the resulting vector.
When $p=q$, the $\ell_{p,q}$ norm coincides with the $\ell_p$ norm, i.e., $\|x\|_{p,p}=\|x\|_p$.
Furthermore, all $\ell_{p,0}$ norms share the same formula, i.e., $\|x\|_{p,0}=\|x\|_{2,0}$, for all $p>0$.
In particular, when the grouping structure is degenerated to the individual feature level, i.e., $n_{\max}=1$ or $n=r$,
we have $\|x\|_{p,q}=\|x\|_q$  for all $p>0$ and $q>0$.

Moreover, we assume that $A$ and $b$ in \eqref{eq-GSOP} are related by a linear model (noiseless)
\[
b=A\bar{x}.
\]
Let $\mathcal{S}:=\left\{i\in \{1,\dots,r\}:\bar{x}_{\mathcal{G}_i}\neq 0\right\}$ be the index set of nonzero groups of $\bar{x}$, $\mathcal{S}^c:=\{1,\dots,r\}\setminus \mathcal{S}$ be the complement of $\mathcal{S}$, $S:=|\mathcal{S}|$ be the group sparsity of $\bar{x}$,
and $n_{\mathbf{a}}:=\sum_{i\in \mathcal{S}} n_i$.

\subsection{Inequalities of $\ell_{p,q}$ norm}
We begin with some basic inequalities of $\ell_p$ and $\ell_{p,q}$ norms, which will be useful in the later discussion of RECs and recovery bounds.
First, we recall the following well-known inequality
\begin{equation}\label{eq-Jensen}
\left(\sum_{i=1}^n |x_i|^{\gamma_2}\right)^{1/{\gamma_2}}\le \left(\sum_{i=1}^n |x_i|^{\gamma_1}\right)^{1/{\gamma_1}}\quad {\rm if}~0<{\gamma_1}\le {\gamma_2},
\end{equation}
or equivalently ($x=(x_1,x_2,\dots,x_n)^\top$),
\begin{equation*}\label{eq-Jensen-norm}
\|x\|_{\gamma_2}\le \|x\|_{\gamma_1}\quad {\rm if}~0<{\gamma_1}\le {\gamma_2}.
\end{equation*}


The following lemma improves \cite[Lemma 4.1]{HuangYang03} and extends to the $\ell_{p,q}$ norm.
It will be useful in providing a shaper global recovery bound (see Theorem \ref{thm-RecoverBound} below).
\begin{lemma}\label{lem-IneQ-2}
Let $0<q\le p\le  2$, $x \in \R^n$ and $K$ be the smallest integer such that $2^{K-1}q\ge 1$.
Then the following relations hold.
\begin{enumerate}[{\rm (i)}]
  \item $\|x\|_q^q\le n^{1-2^{-K}} \|x\|_2^q.$
  \item $\|x\|_{p,q}^q\le r^{1-2^{-K}} \|x\|_{p,2}^q.$
\end{enumerate}
\end{lemma}
\begin{proof}
\begin{enumerate}[{\rm (i)}]
  \item Repeatedly using the property that $\|x\|_1\le \sqrt{n}\|x\|_2$, one has that
\begin{equation*}
\begin{array}{llll}
\|x\|_q^q&\le \sqrt{n} \left(\sum_{i=1}^n|x_i|^{2q}\right)^{1/2}\\
&\le \dots\\
&\le n^{\frac{1}{2}+\dots+\frac{1}{2^K}}\left(\sum_{i=1}^n|x_i|^{2^K q}\right)^{2^{-K}}
\end{array}
\end{equation*}
Since $2^{K-1}q\ge 1$, by \eqref{eq-Jensen}, we have
\[
\left(\sum_{i=1}^n|x_i|^{2^K q}\right)^{2^{-K}}=\left(\sum_{i=1}^n(|x_i|^2)^{2^{K-1} q}\right)^{\frac{1}{2^{K-1}q}\frac{q}{2}}
\le \left(\sum_{i=1}^n|x_i|^2\right)^{q/2}=\|x\|_2^q.
\]
Therefore, we arrive at the conclusion that
\[\|x\|_q^q\le n^{1-2^{-K}}\|x\|_2^q.\]
  \item By \eqref{eq-Lpq norm}, it is a direct consequence of (i).
\end{enumerate}
\end{proof}

For example, if $q=1$, then $K=1$; if $q=\frac12$ or $\frac23$, then $K=2$.
The following lemma describes the triangle inequality of $\|\cdot\|_{p,q}^q$.
\begin{lemma}\label{lem-IneQ-1-a}
Let $0<q\le 1\le p$ and $x,y \in \R^n$.
Then
\[
\|x\|_{p,q}^q-\|y\|_{p,q}^q\le \|x-y\|_{p,q}^q.
\]
\end{lemma}
\begin{proof}
By the subadditivity of $\ell_p$ norm and \eqref{eq-Jensen}, it is easy to see that
\[
\|x_{\mathcal{G}_i}\|_p^q-\|y_{\mathcal{G}_i}\|_p^q\le \|x_{\mathcal{G}_i}-y_{\mathcal{G}_i}\|_p^q,\quad {\rm for}~ i=1,\dots,r.
\]
Consequently, the conclusion directly follows from \eqref{eq-Lpq norm}.
\end{proof}

The following lemma will be beneficial to studying properties of the lower-order REC in Proposition \ref{prop-REC-relation}.
\begin{lemma}\label{lem-IneQ-3}
Let $\gamma\ge 1$, and two finite sequences $\{y_i:i\in \mathcal{I}\}$ and $\{x_j:j\in \mathcal{J}\}$ satisfy that $y_i\ge x_j\ge 0$ for all $(i,j)\in \mathcal{I} \times \mathcal{J}$.
If $\sum_{i\in \mathcal{I}} y_i\ge \sum_{j\in \mathcal{J}} x_j$, then $\sum_{i\in \mathcal{I}} y_i^\gamma \ge \sum_{j\in \mathcal{J}} x_j^\gamma$.
\end{lemma}
\begin{proof}
Set $\bar{y}:=\frac{1}{|\mathcal{I}|}\sum_{i\in \mathcal{I}} y_i$ and $\alpha:=\min_{i\in \mathcal{I}} y_i$. By \cite[Lemma 4.1(ii)]{HuangYang03}, one has that
\begin{equation}\label{eq-lem-IneQ-1}
\sum_{i\in \mathcal{I}} y_i^\gamma\ge \frac{1}{|\mathcal{I}|^{\gamma-1}}\left(\sum_{i\in \mathcal{I}} y_i\right)^\gamma=|\mathcal{I}|\bar{y}^\gamma.
\end{equation}
On the other hand, let $M\in \mathbb{N}$ and $\beta\in [0,\alpha)$ be such that $\sum_{j\in \mathcal{J}} x_j=M\alpha+\beta$.
Observing $\gamma\ge 1$ and $0\le x_j\le \alpha$ for all $j\in \mathcal{J}$, we obtain that $x_j^\gamma\le x_j \alpha^{\gamma-1}$, and thus,
$\sum_{j\in \mathcal{J}}x_j^\gamma\le M\alpha^\gamma+\alpha^{\gamma-1}\beta$.
By \eqref{eq-lem-IneQ-1}, it remains to show that
\begin{equation}\label{eq-lem-IneQ-2}
|\mathcal{I}|\bar{y}^\gamma\ge M\alpha^\gamma+\alpha^{\gamma-1}\beta.
\end{equation}
If $|\mathcal{I}|> M$, the relation \eqref{eq-lem-IneQ-2} is trivial since $\bar{y}\ge \alpha> \beta$;
otherwise, $|\mathcal{I}|\le M$, from the facts that $|\mathcal{I}|\bar{y}\ge M\alpha+\beta$ (i.e., $\sum_{i\in \mathcal{I}} y_i\ge \sum_{j\in \mathcal{J}} x_j$) and that $\gamma\ge 1$, it follows that
\[
|\mathcal{I}|\bar{y}^\gamma\ge M^{1-\gamma}(M\alpha+\beta)^\gamma\ge M^{1-\gamma}(M^\gamma\alpha^\gamma+\gamma M^{\gamma-1}\alpha^{\gamma-1}\beta)
\ge M\alpha^\gamma+\alpha^{\gamma-1}\beta.
\]
Therefore, we obtain the relation \eqref{eq-lem-IneQ-2}, and the proof is complete.
\end{proof}

\subsection{Group restricted eigenvalue conditions}
This subsection aims at the development of the critical conditions on the matrix $A$ to guarantee the oracle property and the global recovery bound of the $\ell_{p,q}$ regularization problem \eqref{eq-GSOP}. In particular, we will focus on the restricted eigenvalue condition (REC), and extend it to the lower-order setting and equip it with the group structure.

In the scenario of sparse optimization, given the sparsity level $s$, it is always assumed that the $2s$-sparse minimal eigenvalue of $A^\top A$ is positive (see, e.g., \cite{Bickel09,Bunea07,Meinshausen09}), that is,
\begin{equation}\label{eq-REC0}
\phi_{\min}(2s):=\min_{\|x\|_0\le 2s} \frac{x^\top A^\top Ax}{x^\top x}>0,
\end{equation}
which is the minimal eigenvalue of any $2s\times 2s$ dimensional submatrix.
It is well-known that the solution at sparsity level $s$ of the linear system $Ax=b$ is unique if the condition \eqref{eq-REC0} is satisfied;
otherwise, assume that there are two distinct vectors $\hat{x}$ and $\tilde{x}$ such that $A\hat{x}=A\tilde{x}$ and $\|\hat{x}\|_0=\|\tilde{x}\|_0=s$.
Then such $x:=\hat{x}-\tilde{x}$ is a vector so that $Ax=0$ and $\|x\|_0\le 2s$, and thus $\phi_{\min}(2s)=0$, which is contradict with \eqref{eq-REC0}.
Therefore, if the $2s$-sparse minimal eigenvalue of $A^\top A$  is zero (i.e., $\phi_{\min}(2s)=0$), one has no hope of recovering the true sparse solution from noisy observations.

However, only the condition \eqref{eq-REC0} is not enough and some further condition is required to maintain the nice recovery of the regularization problem; see \cite{Bickel09,Bunea07,Meinshausen09,Geer09,TZhang09} and references therein.
For example, the REC is introduced in Bickel et al. \cite{Bickel09} to investigate the $\ell_2$ consistency of the $\ell_1$ regularization problem (Lasso), where the minimum in \eqref{eq-REC0} is replaced by the minimum over a restricted set of vectors measured by an $\ell_1$ norm inequality and the denominator is replaced by the $\ell_2$ norm of only a part of $x$.

We now introduce the lower-order REC.
Note that the residual $\hat{x}:=x^*(\ell_q)-\bar{x}$, where $x^*(\ell_q)$ is an optimal solution of $\ell_q$ regularization problem and $\bar{x}$ is a sparse solution of $Ax=b$, of the $\ell_q$ regularization problem always satisfies
\begin{equation}\label{eq-RECq}
\|\hat{x}_{\mathcal{S}^c}\|_q\le \|\hat{x}_{\mathcal{S}}\|_q,
\end{equation}
where $\mathcal{S}$ is the support of $\bar{x}$.
Thus we introduce a lower-order REC, where the minimum is taken over a restricted set measured by an $\ell_q$ norm inequality such as \eqref{eq-RECq},
for establishing the global recovery bound of the $\ell_q$ regularization problem.
Given $s\le t \ll n$, $x\in \R^n$ and $\mathcal{I}\subset\{1,\dots,n\}$, we denote by $\mathcal{I}(x;t)$ the subset of $\{1,\dots,n\}$
corresponding to the first $t$ largest coordinates in absolute value of $x$ in $\mathcal{I}^c$.
\begin{definition}
Let $0\le q\le 1$. The $q$-restricted eigenvalue condition relative to $(s,t)$ ($q$-{\rm REC}$(s,t)$) is said to be satisfied if
\[
\phi_q(s,t):=\min \left\{\frac{\|Ax\|_2}{\|x_{\mathcal{T}}\|_{2}}:|\mathcal{I}|\le s,\|x_{{\mathcal{I}^c}}\|_q\le \|x_{\mathcal{I}}\|_q, \mathcal{T}=\mathcal{I}(x;t)\cup \mathcal{I} \right\}>0.
\]
\end{definition}
The $q$-REC describes a kind of restricted positive definiteness of $A^\top A$, which is valid only for the vectors satisfying the relation measured by an $\ell_q$ norm. The $q$-REC presents a unified framework of the REC-type conditions when $q\in [0,1]$. In particular, we note by definition that $1$-REC reduces to the classical REC \cite{Bickel09}, and that $\phi_{\min}(2s)=\phi_0^2(s,s)$, and thus
\[
\mbox{\eqref{eq-REC0}\quad $\Leftrightarrow$ \quad $0$-REC$(s,s)$ is satisfied.}
\]

It is well-known in the literature that the 1-REC is a stronger condition than the 0-REC (i.e., \eqref{eq-REC0}).
A natural question arises what are the relationship between the general $q$-RECs.
To answer this question, associated with the $q$-REC, we consider the feasible set
\[
C_q(s):=\{x\in \R^n: \|x_{{\mathcal{I}^c}}\|_q\le \|x_{\mathcal{I}}\|_q~{\rm for}~{\rm some}~ |\mathcal{I}|\le s\},
\]
 which is a cone.
Since the objective function associated with the $q$-REC is homogeneous, the $q$-REC$(s,t)$ says that the null space of $A$ does not cross over $C_q(s)$.
Figure \ref{fig-RECs} presents the geometric interpretation of the $q$-RECs. It is shown that $C_0(s)\subseteq C_{1/2}(s) \subseteq C_1(s)$, and thus
\[
\mbox{1-REC $\Rightarrow$ 1/2-REC $\Rightarrow$ 0-REC}.
\]
It is also observed from Figure \ref{fig-RECs} that the gap between the 1-REC and 1/2-REC and that between 1/2-REC and 0-REC are the matrices whose null spaces all fall in the cones of $C_1(s)\setminus C_{1/2}(s)$ and $C_{1/2}(s)\setminus C_0(s)$, respectively.


We now provide a rigorous proof in the following proposition to identify the relationship between the feasible sets $C_{q}(s)$ and between the general $q$-RECs: the lower the $q$, the smaller the cone $C_{q}(s)$, and the weaker the $q$-REC.

\begin{proposition}\label{prop-REC-relation}
Let $0\le q_1\le q_2\le 1$ and $1\le s\le t\ll n$. Then the following statements are true:
\begin{enumerate}[{\rm (i)}]
  \item $C_{q_1}(s)\subseteq C_{q_2}(s)$, and
  \item if the $q_2$-{\rm REC}$(s,t)$ holds, then the $q_1$-{\rm REC}$(s,t)$ holds.
\end{enumerate}
\end{proposition}
\begin{proof}
\begin{enumerate}[{\rm (i)}]
\item Let $x\in C_{q_1}(s)$.
We use $\mathcal{I}_*$ to denote the index set of the first $s$ largest
coordinates in absolute value of $x$. Since $x\in C_{q_1}(s)$, it follows that $\|x_{\mathcal{I}_*^c}\|_{q_1}\le \|x_{\mathcal{I}_*}\|_{q_1}$ ($|\mathcal{I}_*|\le s$ due to the construction of $\mathcal{I}_*$). By Lemma \ref{lem-IneQ-3} (taking $\gamma=q_2/q_1$), one has that
\[
\|x_{\mathcal{I}_*^c}\|_{q_2}\le \|x_{\mathcal{I}_*}\|_{q_2},
\]
that is, $x\in C_{q_2}(s)$. Hence it follows that $C_{q_1}(s)\subseteq C_{q_2}(s)$.
\item As proved by (i) that $C_{q_1}(s)\subseteq C_{q_2}(s)$, by the definition of $q$-REC, it follows that
\[
\phi_{q_1}(s,t)\ge \phi_{q_2}(s,t)>0.
\]
The proof is complete.
\end{enumerate}
\end{proof}

To the best of our knowledge, this is the first work on introducing the lower-order REC and establishing the relationship of the lower-order RECs.
In the following, we provide a counter example to show that the reverse of Proposition \ref{prop-REC-relation} is not true.
\begin{example}[A matrix satisfying $1/2$-REC but not REC]\label{ex-REC}
Consider the matrix
\[
A=\left(\begin{matrix}
   a &a+c &a-c \\
   \tilde{a} &\tilde{a}-\tilde{c} &\tilde{a}+\tilde{c}
  \end{matrix}\right)\in \R^{2\times 3},
\]
where $a>c>0$ and $\tilde{a}>\tilde{c}>0$.
This matrix $A$ does not satisfy the REC(1,1). Indeed, by letting $\mathcal{J}=\{1\}$ and $x=(2,-1,-1)^\top$,
we have $Ax=0$ and thus $\phi(1,1)=0$.

However, $A$ satisfies the $1/2$-REC(1,1). It suffices to show that $\phi_{1/2}(1,1)>0$.
Let $x=(x_1,x_2,x_3)^\top$ satisfy the constraint associated with $1/2$-REC(1,1).
As $s=1$, the deduction is divided into the following three cases.
\begin{enumerate}[{\rm (i)}]
  \item $\mathcal{J}=\{1\}$. Then
      \begin{equation}\label{eq-exp-0}
      |x_1|\ge \|x_{\mathcal{J}^c}\|_{1/2}=|x_2|+|x_3|+2|x_2|^{1/2}|x_3|^{1/2}.
      \end{equation}
      Without loss of generality, we assume $|x_1|\ge |x_2|\ge|x_3|$. Hence, $\mathcal{T}=\{1,2\}$ and
      \begin{equation}\label{eq-exp-2}
      \frac{\|Ax\|_2}{\|x_{\mathcal{T}}\|_2}\ge \frac{\min\{a,\tilde{a}\}|x_1+x_2+x_3|+\min\{c,\tilde{c}\}|x_2-x_3|}{|x_1|+|x_2|}.
      \end{equation}
      If $|x_2|\le \frac13|x_1|$, \eqref{eq-exp-2} reduces to
      \begin{equation}\label{eq-exp-a1}
      \frac{\|Ax\|_2}{\|x_{\mathcal{T}}\|_2}\ge\frac{\frac{\min\{a,\tilde{a}\}}{3}|x_1|}{\frac{4}{3}|x_1|}=\frac{\min\{a,\tilde{a}\}}{4}.
      \end{equation}
      If $|x_2|\ge \frac13|x_1|$, substituting \eqref{eq-exp-0} into \eqref{eq-exp-2}, one has that
      \begin{equation}\label{eq-exp-a2}
      \frac{\|Ax\|_2}{\|x_{\mathcal{T}}\|_2}\ge \frac{2\min\{a,\tilde{a}\}|x_2|^{1/2}|x_3|^{1/2}+\min\{c,\tilde{c}\}|x_2-x_3|}{4|x_2|}
      \ge\left\{\begin{matrix}
      \frac{\min\{c,\tilde{c}\}}{8},&{|x_3|\le \frac{1}{2}|x_2|,}\\ \frac{\min\{a,\tilde{a}\}}{4}, &|x_3|\ge \frac{1}{2}|x_2|.
      \end{matrix}\right.
      \end{equation}
  \item $\mathcal{J}=\{2\}$. Since $\mathcal{T}=\{2,1\}$ or $\{2,3\}$, it follows from \cite[Lemma 4.1(ii)]{HuangYang03} that
      \begin{equation}\label{eq-exp-1}
      |x_2|\ge \|x_{\mathcal{J}^c}\|_{1/2}\ge |x_1|+|x_3|.
      \end{equation}
      Thus, it is easy to verify that $\|x_{\mathcal{T}}\|_2\le 2|x_2|$ and that
      \begin{equation}\label{eq-exp-a3}
      \frac{\|Ax\|_2}{\|x_{\mathcal{T}}\|_2}\ge \frac{|ax_1+(a+c)x_2+(a-c)x_3|}{2|x_2|}=\frac{|a(x_1+x_2+\frac{a-c}{a}x_3)+cx_2|}{2|x_2|}\ge \frac{c}{2},
      \end{equation}
      where the last inequality follows from \eqref{eq-exp-1} and the fact that $a>c$.
  \item $\mathcal{J}=\{3\}$. Similar to the deduction of {\rm (ii)}, we have that
      \begin{equation}\label{eq-exp-b3}
      \frac{\|Ax\|_2}{\|x_{\mathcal{T}}\|_2}\ge \frac{|\tilde{a}x_1+(\tilde{a}-\tilde{c})x_2+(\tilde{a}+\tilde{c})x_3|}{2|x_3|}\ge \frac{\tilde{c}}{2}.
      \end{equation}
\end{enumerate}
 Therefore, by \eqref{eq-exp-a1}, \eqref{eq-exp-a2}, \eqref{eq-exp-a3} and \eqref{eq-exp-b3}, one has that
      $\phi_{1/2}(1,1)\ge \frac18\min\{c,\tilde{c}\}>0$, and thus, the matrix $A$ satisfies the $1/2$-REC(1,1).
\end{example}

In order to establish the oracle property and the global recovery bound for the $\ell_{p,q}$ regularization problem,
we further introduce the notion of group restricted eigenvalue condition (GREC).
Given $S\le N\ll r$, $x\in \R^n$ and $\mathcal{J}\subset\{1,\dots,r\}$,
we use ${\rm rank}_i(x)$ to denote the rank of $\|x_{\mathcal{G}_i}\|_p$ among $\{\|x_{\mathcal{G}_j}\|_p: j\in\mathcal{J}^c\}$ (in a decreasing order),
$\mathcal{J}(x;N)$ to denote the index set of the first $N$ largest groups in the value of $\|x_{\mathcal{G}_i}\|_p$ among $\{\|x_{\mathcal{G}_j}\|_p: j\in\mathcal{J}^c\}$, that is,
\[
\mathcal{J}(x;N):=\left\{i\in\mathcal{J}^c:{\rm rank}_i(x)\in\{1,\dots,N\}\right\}.
\]
Furthermore, by letting $R:=\lceil \frac{r-|\mathcal{J}|}{N} \rceil$, we denote
\begin{equation}\label{eq-Jk}
\mathcal{J}_k(x;N):=\left\{\begin{matrix}
   \left\{i\in\mathcal{J}^c:{\rm rank}_i(x)\in\{kN+1,\dots,(k+1)N\}\right\},&{k=1,\dots,R-1,}\\
   \left\{i\in\mathcal{J}^c:{\rm rank}_i(x)\in\{RN+1,\dots,r-|\mathcal{J}|\}\right\}, &k=R.
\end{matrix}\right.
\end{equation}


Note that the residual $\hat{x}:=x^*(\ell_{p,q})-\bar{x}$ of the $\ell_{p,q}$ regularization problem always satisfies
$\|\hat{x}_{\mathcal{G}_{\mathcal{S}^c}}\|_{p,q}\le \|\hat{x}_{\mathcal{G}_\mathcal{S}}\|_{p,q}$.
Thus we introduce the notion of GREC, where the minimum is taken over a restricted set measured by an $\ell_{p,q}$ norm inequality, as follows.
\begin{definition}\label{def-GREC}
Let $0<q\le p\le 2$.
The $(p,q)$-group restricted eigenvalue condition relative to $(S,N)$ ($(p,q)$-{\rm GREC}$(S,N)$) is said to be satisfied if
\[
\phi_{p,q}(S,N):=\min \left\{\frac{\|Ax\|_2}{\|x_{\mathcal{G}_\mathcal{N}}\|_{p,2}}:|\mathcal{J}|\le S,\|x_{\mathcal{G}_{\mathcal{J}^c}}\|_{p,q}\le \|x_{\mathcal{G}_\mathcal{J}}\|_{p,q}, \mathcal{N}=\mathcal{J}(x;N)\cup \mathcal{J} \right\}>0.
\]
\end{definition}
The $(p,q)$-GREC extends the $q$-REC to the setting equipping with a pre-defined group structure.
Handling the components in each group as one element, the $(p,q)$-GREC admits the fewer degree of freedom, which is $S$, about $s/n_{\max}$, on its associated constraint than that of the $q$-REC, and thus it characterizes a weaker condition than the $q$-REC.
For example, the $0$-REC$(s,s)$ is to indicate the restricted positive definiteness of $A^\top A$, which is valid only for the vectors whose cardinality is less than $2s$; while the $(p,0)$-GREC$(S,S)$ is to describe the restricted positive definiteness of $A^\top A$ on any $2S$-group support, whose degree of freedom is much less than the $2s$-support. Thus the $(p,0)$-GREC$(S,S)$ provides a broader condition than the $0$-REC$(s,s)$.
Similar to the proof of Proposition \ref{prop-REC-relation}, we can show that
if $0\le q_1\le q_2\le 1\le p\le 2$ and the $(p,q_2)$-{\rm GREC}$(S,N)$ holds, then the $(p,q_1)$-{\rm GREC}$(S,N)$ also holds.

We end this subsection by providing the following lemma, which will be useful in establishing
the global recovery bound for the $\ell_{p,q}$ regularization problem in Theorem \ref{thm-RecoverBound}.
\begin{lemma}\label{lem-bound-Nc}
Let $0<q\le 1\le p$, $\tau\ge 1$ and $x\in \R^n$, $\mathcal{N}:=\mathcal{J}(x;N)\cup \mathcal{J}$ and $\mathcal{J}_k:=\mathcal{J}_k(x;N)$ for $k=1,\dots,R$. Then the following inequalities hold
\[
\|x_{\mathcal{G}_{\mathcal{N}^c}}\|_{p,\tau}\le \sum_{k=1}^R\|x_{\mathcal{G}_{\mathcal{J}_k}}\|_{p,\tau}\le N^{\frac{1}{\tau}-\frac{1}{q}} \|x_{\mathcal{G}_{\mathcal{J}^c}}\|_{p,q}.
\]
\end{lemma}
\begin{proof}
By the definition of $\mathcal{J}_k$ (cf. \eqref{eq-Jk}), for all $j\in \mathcal{J}_{k}$, one has that
\[
\|x_{\mathcal{G}_j}\|_p\le \|x_{\mathcal{G}_i}\|_p,\quad  {\rm for}~ i\in \mathcal{J}_{k-1},
\]
and thus
\[
\|x_{\mathcal{G}_j}\|_p^q\le \frac{1}{N}\sum_{i\in \mathcal{J}_{k-1}}\|x_{\mathcal{G}_i}\|_p^q=\frac{1}{N}\|x_{\mathcal{G}_{\mathcal{J}_{k-1}}}\|_{p,q}^q.
\]
Consequently, we obtain that
\[
\|x_{\mathcal{G}_{\mathcal{J}_k}}\|_{p,\tau}^\tau=\sum_{i\in \mathcal{J}_k} \|x_{\mathcal{G}_i}\|_p^\tau\le N^{1-\tau/q} \|x_{\mathcal{G}_{\mathcal{J}_{k-1}}}\|_{p,q}^\tau.
\]
Further by \cite[Lemma 4.1]{HuangYang03} ($\tau\ge 1$ and $q\le 1$), it follows that
\[
\begin{array}{llll}
\|x_{\mathcal{G}_{\mathcal{N}^c}}\|_{p,\tau}&= \left(\sum_{k=1}^R \sum_{i\in \mathcal{J}_k}\|x_{\mathcal{G}_{i}}\|_{p}^\tau\right)^{1/\tau}
\le \sum_{k=1}^R\|x_{\mathcal{G}_{\mathcal{J}_k}}\|_{p,\tau}\\
&\le N^{\frac{1}{\tau}-\frac{1}{q}}\sum_{k=1}^R\|x_{\mathcal{G}_{\mathcal{J}_{k-1}}}\|_{p,q}
\le N^{\frac{1}{\tau}-\frac{1}{q}}\|x_{\mathcal{G}_{\mathcal{J}^c}}\|_{p,q}.
\end{array}
\]
The proof is complete.
\end{proof}

\subsection{Global recovery bound}
In recent years, many articles have been devoted to establishing the oracle property and the global recovery bound for the $\ell_1$ regularization problem \eqref{eq-L1} under the RIP or REC; see, e.g., \cite{Bickel09,Bunea07,Meinshausen09,Geer09,TZhang09}.
However, to the best of our knowledge, few papers concentrate on investigating these properties for the lower-order regularization problem.

In the preceding subsections, we have introduced the general notion of $(p,q)$-GREC.
Under the $(p,q)$-{\rm GREC}$(S,S)$, the solution of $Ax=b$ with group sparsity being $S$ is unique.
In this subsection, we will present the oracle property and the global recovery bound for the $\ell_{p,q}$ regularization problem \eqref{eq-GSOP} under the $(p,q)$-GREC.
The oracle property provides an upper bound on the square error of the linear system and the violation of the true nonzero groups for each point in
the level set of the objective function of \eqref{eq-GSOP}
\begin{equation*}\label{eq-level}
{\rm lev}_F(\bar{x}):=\{x\in \R^n: \|Ax-b\|_2^2+\lambda\|x\|_{p,q}^q\le \lambda\|\bar{x}\|_{p,q}^q\}.
\end{equation*}

\begin{proposition}\label{prop-OI}
Let $0<q\le 1\le p$, $S>0$ and let the $(p,q)$-{\rm GREC}$(S,S)$ hold.
Let $\bar{x}$ be the unique solution of $Ax=b$ at a group sparsity level $S$, and $\mathcal{S}$ be the index set of nonzero groups of $\bar{x}$.
Let $K$ be the smallest integer such that $2^{K-1}q\ge 1$. Then, for any $x^*\in {\rm lev}_F(\bar{x})$, the following oracle inequality holds
\begin{equation}\label{eq-OI}
\|Ax^*-A\bar{x}\|_2^2+\lambda\|x_{\mathcal{G}_{\mathcal{S}^c}}^*\|_{p,q}^q\le \lambda^\frac{2}{2-q}S^{\left(1-2^{-K}\right)\frac{2}{2-q}}/\phi_{p,q}^{\frac{2q}{2-q}}(S,S).
\end{equation}
Moreover, letting ${\mathcal{N}_*}:=\mathcal{S}\cup \mathcal{S}(x^*;S)$, we have
\[
\|x_{\mathcal{G}_{\mathcal{N}_*}}^*-\bar{x}_{\mathcal{G}_{\mathcal{N}_*}}\|_{p,2}^2\le \lambda^\frac{2}{2-q}S^{\left(1-2^{-K}\right)\frac{2}{2-q}}/\phi_{p,q}^{\frac{4}{2-q}}(S,S).
\]
\end{proposition}
\begin{proof}
Let $x^*\in {\rm lev}_F(\bar{x})$. That is, $\|Ax^*-b\|_2^2+\lambda\|x^*\|_{p,q}^q\le \lambda\|\bar{x}\|_{p,q}^q$. By Lemmas \ref{lem-IneQ-2}(ii) and \ref{lem-IneQ-1-a}, one has that
\begin{equation}\label{eq-lem-OI-1}
\begin{array}{llll}
\|Ax^*-A\bar{x}\|_2^2+\lambda \|x_{\mathcal{G}_{\mathcal{S}^c}}^*\|_{p,q}^q&\le \lambda\|\bar{x}_{\mathcal{G}_\mathcal{S}}\|_{p,q}^q-\lambda\|x_{\mathcal{G}_\mathcal{S}}^*\|_{p,q}^q\\
&\le \lambda \|\bar{x}_{\mathcal{G}_\mathcal{S}}-x_{\mathcal{G}_\mathcal{S}}^*\|_{p,q}^q\\
&\le \lambda S^{1-2^{-K}} \|\bar{x}_{\mathcal{G}_\mathcal{S}}-x_{\mathcal{G}_\mathcal{S}}^*\|_{p,2}^q.
\end{array}
\end{equation}
Noting that
\begin{equation*}
\|x_{\mathcal{G}_{\mathcal{S}^c}}^*-\bar{x}_{\mathcal{G}_{\mathcal{S}^c}}\|_{p,q}^q-\|x_{\mathcal{G}_\mathcal{S}}^*-\bar{x}_{\mathcal{G}_\mathcal{S}}\|_{p,q}^q\le \|x_{\mathcal{G}_{\mathcal{S}^c}}^*\|_{p,q}^q-(\|\bar{x}_{\mathcal{G}_\mathcal{S}}\|_{p,q}^q-\|x_{\mathcal{G}_\mathcal{S}}^*\|_{p,q}^q)=\|x^*\|_{p,q}^q-\|\bar{x}\|_{p,q}^q\le 0.
\end{equation*}
Then $(p,q)$-{\rm GREC}$(S,S)$ implies that
\begin{equation}\label{eq-lem-OI-2}
\|\bar{x}_{\mathcal{G}_\mathcal{S}}-x_{\mathcal{G}_\mathcal{S}}^*\|_{p,2}\le\|Ax^*-A\bar{x}\|_2/\phi_{p,q}(S,S).
\end{equation}
From \eqref{eq-lem-OI-1} and \eqref{eq-lem-OI-2}, it follows that
\begin{equation}\label{eq-lem-OI-3}
\|Ax^*-A\bar{x}\|_2^2+\lambda \|x_{\mathcal{G}_{\mathcal{S}^c}}^*\|_{p,q}^q\le \lambda S^{1-2^{-K}}\|Ax^*-A\bar{x}\|_2^q/\phi_{p,q}^q(S,S),
\end{equation}
and consequently,
\begin{equation}\label{eq-lem-OI-4}
\|Ax^*-A\bar{x}\|_2\le \lambda^{\frac{1}{2-q}}S^{\left(1-2^{-K}\right)/(2-q)}/\phi_{p,q}^{\frac{q}{2-q}}(S,S).
\end{equation}
Therefore, by \eqref{eq-lem-OI-3} and \eqref{eq-lem-OI-4}, we arrive at the oracle inequality \eqref{eq-OI}.
Furthermore, by the definition of $\mathcal{N}_*$, $(p,q)$-{\rm GREC}$(S,S)$ implies that
\[
\begin{array}{llll}
\|x_{\mathcal{G}_{\mathcal{N}_*}}^*-\bar{x}_{\mathcal{G}_{\mathcal{N}_*}}\|_{p,2}^2
\le \|Ax^*-A\bar{x}\|_2^2/\phi_{p,q}^2(S,S) \le \lambda^\frac{2}{2-q}S^{\left(1-2^{-K}\right)\frac{2}{2-q}}/\phi_{p,q}^{\frac{4}{2-q}}(S,S).
\end{array}
\]
The proof is complete.
\end{proof}

One of the main results of this section is presented as follows, where we establish the global recovery bound for the $\ell_{p,q}$ regularization problem under the $(p,q)$-GREC. We will apply oracle inequality \eqref{eq-OI} and Lemma \ref{lem-bound-Nc} in our proof.

\begin{theorem}\label{thm-RecoverBound}
Let $0<q\le 1\le p\le 2$, $S>0$ and let the $(p,q)$-{\rm GREC}$(S,S)$ hold.
Let $\bar{x}$ be the unique solution of $Ax=b$ at a group sparsity level $S$, and $\mathcal{S}$ be the index set of nonzero groups of $\bar{x}$.
Let $K$ be the smallest integer such that $2^{K-1}q\ge 1$. Then, for any $x^*\in {\rm lev}_F(\bar{x})$, the following global recovery bound for \eqref{eq-GSOP} holds
\begin{equation}\label{eq-RB-1}
\|x^*-\bar{x}\|_2^2\le 2\lambda^\frac{2}{2-q}S^{\frac{q-2}{q}+\left(1-2^{-K}\right)\frac{4}{q(2-q)}}/\phi_{p,q}^{\frac{4}{2-q}}(S,S).
\end{equation}
More precisely,
\begin{equation}\label{eq-RB-Lp,q}
\|x^*-\bar{x}\|_2^2\le\left\{\begin{matrix}
   O\left(\lambda^\frac{2}{2-q}S\right),&{2^{K-1}q=1,}\\
   O\left(\lambda^\frac{2}{2-q}S^\frac{3-q}{2-q}\right), &{2^{K-1}q>1.}
\end{matrix}\right.
\end{equation}
\end{theorem}
\begin{proof}
Let ${\mathcal{N}_*}:=\mathcal{S}\cup \mathcal{S}(x^*;S)$ as in Proposition \ref{prop-OI}. Since $p\le 2$, it follows from Lemma \ref{lem-bound-Nc} and Proposition \ref{prop-OI} that
\[
\|x_{\mathcal{G}_{\mathcal{N}_*^c}}^*\|_2^2\le \|x_{\mathcal{G}_{\mathcal{N}_*^c}}^*\|_{p,2}^2\le S^{1-2/q}\|x_{\mathcal{G}_{\mathcal{S}^c}}^*\|_{p,q}^2\le \lambda^\frac{2}{2-q}S^{\frac{q-2}{q}+\left(1-2^{-K}\right)\frac{4}{q(2-q)}}/\phi_{p,q}^{\frac{4}{2-q}}(S,S).
\]
Then by Proposition \ref{prop-OI}, one has that
\[
\begin{array}{llll}
\|x^*-\bar{x}\|_2^2&= \|x_{\mathcal{G}_{\mathcal{N}_*}}^*-\bar{x}_{\mathcal{G}_{\mathcal{N}_*}}\|_2^2+\|x_{\mathcal{G}_{\mathcal{N}_*^c}}^*\|_2^2\\
&\le \lambda^\frac{2}{2-q}S^{\left(1-2^{-K}\right)\frac{2}{2-q}}/\phi_{p,q}^{\frac{4}{2-q}}(S,S)+
\lambda^\frac{2}{2-q}S^{\frac{q-2}{q}+\left(1-2^{-K}\right)\frac{4}{q(2-q)}}/\phi_{p,q}^{\frac{4}{2-q}}(S,S)\\
&\le 2\lambda^\frac{2}{2-q}S^{\frac{q-2}{q}+\left(1-2^{-K}\right)\frac{4}{q(2-q)}}/\phi_{p,q}^{\frac{4}{2-q}}(S,S),
\end{array}
\]
where the last inequality follows from the fact that $2^{K-1}q\ge 1$.
In particular, if $2^{K-1}q=1$, then $\frac{q-2}{q}+\left(1-2^{-K}\right)\frac{4}{q(2-q)}=1$ and thus
\[
\|x^*-\bar{x}\|_2^2\le O\left(\lambda^\frac{2}{2-q}S\right).
\]
If $2^{K-1}q>1$, then $2^{K-2}q<1$. Hence, $\frac{q-2}{q}+\left(1-2^{-K}\right)\frac{4}{q(2-q)}<\frac{3-q}{2-q}$, and consequently
\[
\|x^*-\bar{x}\|_2^2\le O\left(\lambda^\frac{2}{2-q}S^\frac{3-q}{2-q}\right).
\]
The proof is complete.
\end{proof}


Theorem \ref{thm-RecoverBound} is an important theoretical result in that it provides the global recovery bound \eqref{eq-RB-Lp,q} for the general $\ell_{p,q}$ regularization problem \eqref{eq-GSOP}. In particular, when $x^*$ is a global optimal solution of \eqref{eq-GSOP} as assumed in \cite{Geer09}, Theorem \ref{thm-RecoverBound} provides an upper bound on the distance from the optimal solution $x^*$ to the true sparse solution $\bar{x}$. This result is significantly different from the previous works \cite{Bickel09,Bunea07,Meinshausen09,TZhang09}, our result here is deterministic and does not involve any kind of randomization, and thus does not have a nonzero probability of failure.
The following two corollaries show the recovery bounds for the group Lasso and $\ell_{p,1/2}$ regularization problem.
\begin{corollary}\label{coro-RB-1}
Let $\bar{x}$ be a solution of $Ax=b$, and $S$ be the group sparsity of $\bar{x}$.
Let $1\le p\le 2$, and let $x^*(\ell_{p,1})$ be an optimal solution of the $\ell_{p,1}$ regularization problem. Suppose that the $(p,1)$-{\rm GREC}$(S,S)$ holds.
Then the following recovery bound for the $\ell_{p,1}$ regularization problem holds
\begin{equation}\label{eq-RB-G1}
\|x^*(\ell_{p,1})-\bar{x}\|_2^2\le O\left(\lambda^2S\right).
\end{equation}
\end{corollary}
Corollary \ref{coro-RB-1} extends the study of the $\ell_1$ regularization problem (Lasso) in \cite{Bickel09,Geer09}, where the recovery bound is given by
\begin{equation}\label{eq-RB-Lasso}
\|x^*(\ell_1)-\bar{x}\|_2^2\le O\left(\lambda^2s\right).
\end{equation}
Comparing with \eqref{eq-RB-Lasso}, Corollary \ref{coro-RB-1} provides the theoretical evidence for the phenomenon that
exploiting the group sparsity structure can enhance the recovery performance when $S\ll s$.

\begin{corollary}\label{coro-RB-2}
Let $\bar{x}$ be a solution of $Ax=b$, and $S$ be the group sparsity of $\bar{x}$.
Let $1\le p\le 2$, and let $x^*(\ell_{p,1/2})$ be a global optimal solution of the $\ell_{p,1/2}$ regularization problem. Suppose that the  $(p,1/2)$-{\rm GREC}$(S,S)$ holds.
Then the following global recovery bound for the $\ell_{p,1/2}$ regularization problem holds
\[
\|x^*(\ell_{p,1/2})-\bar{x}\|_2^2\le O\left(\lambda^{4/3}S\right).
\]
\end{corollary}
In particular, when $n_{max}=1$, that is, the problem is in absence of the group structure, Corollary \ref{coro-RB-2} exhibits the following global recovery bound for the $\ell_{1/2}$ regularization problem
\begin{equation}\label{eq-RB-L1/2}
\|x^*(\ell_{1/2})-\bar{x}\|_2^2\le O\left(\lambda^{4/3}s\right).
\end{equation}
%

Bound \eqref{eq-RB-G1} seems to be the best one in a sense that all other bounds $O\left(\lambda^{\frac{2}{2-q}}S\right)$ are not better than it when $q<1$ and the $(p,1)$-{\rm GREC}$(S,S)$ holds. However, when the $(p,1)$-{\rm GREC}$(S,S)$ does not hold but $(p,1/2)$-{\rm GREC}$(S,S)$ holds, we illustrate by an example that \eqref{eq-RB-G1} or \eqref{eq-RB-Lasso} does not hold but \eqref{eq-RB-L1/2} does and is also tight.
We will testify the recovery bound \eqref{eq-RB-L1/2} by using a global optimization method.

\begin{example}\label{ex-REC-NE}
By letting $a=\tilde{a}=2$ and $c=\tilde{c}=1$ in Example \ref{ex-REC}, we consider the following matrix:
\[
A=\left(\begin{matrix}
   2 &3 &1 \\
   2 &1 &3
  \end{matrix}\right).
\]
We assume $b=(2,2)^\top$ and then a true solution of $Ax=b$ is $\bar{x}=(1,0,0)^\top$. Denoting $x:=(x_1,x_2,x_3)^\top$, the objective function associated with the $\ell_1$ regularization problem \eqref{eq-L1} is
\[
\begin{array}{ll}
F(x)&:=\|Ax-b\|_2^2+\lambda\|x\|_1\\
&=(2x_1+3x_2+x_3-2)^2+(2x_1+x_2+3x_3-2)^2+\lambda(|x_1|+|x_2|+|x_3|).
\end{array}\]
Let $x^*(\ell_1):=(x_1^*,x_2^*,x_3^*)^\top$ be an optimal solution of problem \eqref{eq-L1}.
Without loss of generality, we assume $\lambda\le 1$.
The necessary condition of $x^*(\ell_1)$ being an optimal solution of \eqref{eq-L1} is $0\in \partial F(x^*(\ell_1))$, that is,
\begin{subequations}\label{eq-exp2}
\begin{align}
0&\in  16x_1^*+16x_2^*+16x_3^*-16+\lambda\partial|x_1^*|, \label{eq-exp2-a}\\
0&\in  16x_1^*+20x_2^*+12x_3^*-16+\lambda\partial|x_2^*|, \label{eq-exp2-b}\\
0&\in  16x_1^*+12x_2^*+20x_3^*-16+\lambda\partial|x_3^*|, \label{eq-exp2-c}
\end{align}
\end{subequations}
where
$\partial |\mu|:=\left\{\begin{matrix}
   sgn(\mu),&{\mu\neq 0,}\\ [-1,1], &\mu=0.
\end{matrix}\right.$

We first show that $x_i^*\ge 0$ for $i=1,2,3$ by contradiction. Indeed, if $x_1^*<0$, \eqref{eq-exp2-a} reduces to
\begin{equation*}
16x_1^*+16x_2^*+16x_3^*-16=\lambda.
\end{equation*}
Summing \eqref{eq-exp2-b} and \eqref{eq-exp2-c}, we further have
\[
\lambda=16x_1^*+16x_2^*+16x_3^*-16\in -\frac{\lambda}{2}(\partial|x_2^*|+\partial|x_3^*|),
\]
which implies that $x_2^*\le 0$ and $x_3^*\le 0$. Hence, it follows that $F(x^*)> F(0)$, which indicates that $x^*$ is not an optimal solution of \eqref{eq-L1}, and thus, $x_1^*<0$ is impossible.
Similarly, we can show that $x_2^*\ge 0$ and $x_3^*\ge 0$.

Next, we find the optimal solution $x^*(\ell_1)$ by only considering $x^*(\ell_1)\ge 0$. It is easy to obtain that the solution of \eqref{eq-exp2} and the corresponding objective value associated with \eqref{eq-L1} can be represented respectively by
\[x_1^*=1-\frac{\lambda}{16}-2x_3^*,\quad x_2^*=x_3^*~\left(0\le x_3^*\le \frac12-\frac{\lambda}{32}\right),
\quad {\rm and}\quad F\left(x^*(\ell_1)\right)=\lambda-\frac{\lambda^2}{32}.\]
Hence, $x^*(\ell_1):=\left(0,\frac12-\frac{\lambda}{32},\frac12-\frac{\lambda}{32}\right)^\top$ is an optimal solution of problem \eqref{eq-L1}.
The estimated error for such $x^*(\ell_1)$ is
\[
\|x^*(\ell_1)-\bar{x}\|_2^2=1+\frac12\left(1-\frac{\lambda}{16}\right)^2>1,
\]
which does not meet the recovery bound \eqref{eq-RB-Lasso} for each $\lambda\le 1$.

It is revealed from Example \ref{ex-REC} that this matrix $A$ satisfies the $1/2$-REC(1,1). Then the hypothesis of Corollary \ref{coro-RB-2} is verified, and thus, Corollary \ref{coro-RB-2} is applicable to establishing the recovery bound \eqref{eq-RB-L1/2} for the $\ell_{1/2}$ regularization problem.
Even though we cannot obtain the closed-form solution of this nonconvex $\ell_{1/2}$ regularization problem, as it is of only 3-dimensions, we use a global optimization method, the filled function method \cite{Ge90}, to find the global optimal solution $x^*(\ell_{1/2})$ and thus to testify the recovery bound \eqref{eq-RB-L1/2}.
This is done by computing the $\ell_{1/2}$ regularization problem for many $\lambda$ to plot the curve $\|x^*(\ell_{1/2})-\bar{x}\|_2^2$.
Figure \ref{fig-RB} illustrates the variation of the estimated error $\|x^*(\ell_{1/2})-\bar{x}\|_2^2$ and the bound $2\lambda^{4/3}$ (that is the right-hand side of \eqref{eq-RB-1}, where $S=1$ and $\phi_{1/2}(1,1)\le 1$ (cf. Example \ref{ex-REC})), when varying the regularization parameter $\lambda$ from $10^{-8}$ to $1$.

\begin{figure}[h]
\centering
  \includegraphics[width=8cm]{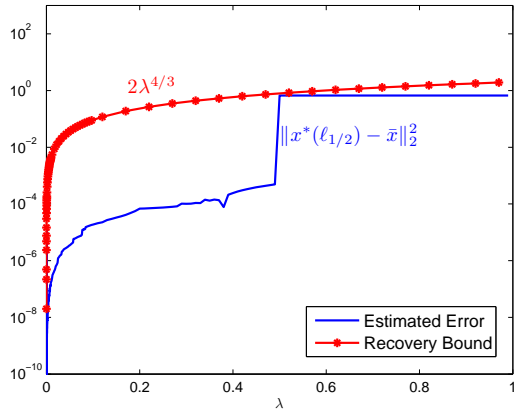}\\
  \caption{The illustration of the recovery bound \eqref{eq-RB-1} and estimated error.}
  \label{fig-RB}
\end{figure}
\end{example}

\subsection{Local recovery bound}
In the preceding subsection, we provided the global analysis of the recovery bound for the $\ell_{p,q}$ regularization problem under the $(p,q)$-GREC; see Theorem \ref{thm-RecoverBound}.
One can also see from Figure \ref{fig-RB} that the global recovery bound \eqref{eq-RB-L1/2} is tight for the $\ell_{1/2}$ regularization problem
as the curves come together at $\lambda\simeq 0.5$, but there is still a big gap for the improvement when $\lambda$ is small.

This subsection is devoted to providing a local analysis of the recovery bound for the $\ell_{p,q}$ regularization problem
by virtue of the technique of variational analysis \cite{Roc98}. For $x\in \R^n$ and $\delta\in\R_+$, $\mathbf{B}(x,\delta)$ denotes the open ball of radius $\delta$ centered at $x$.
For a lower semi-continuous (lsc) function  $f:\R^n\rightarrow \R$ and $x,w\in \R^n$, the subderivative of $f$ at $x$ along the direction $w$ is defined by
\[
df(\bar{x})(w):=\liminf_{\tau\downarrow 0, \;w'\rightarrow w}\frac{f(\bar{x}+\tau w')-f(\bar{x})}{\tau}.
\]
To begin with, we show in the following lemma a significant advantage of lower-order regularization over the $\ell_1$ regularization: the lower-order regularization term can easily induce the sparsity of the local minimum.
\begin{lemma}\label{meng-h-y}
Let $0<q<1\le p$. Let $f:\R^n\rightarrow \R$ be a lsc function and $df(0)(0)=0$. Then
the function  $F:=f+\lambda\|\cdot\|_{p,q}^{q}$ has a local minimum at $0$ with the first-order growth condition being fulfilled, i.e.,  there exist some $\epsilon>0$ and $\delta>0$ such that
\[
F(x)\ge F(0)+\epsilon\|x\|_2\quad\forall x\in \mathbf{B}(0,\delta).
\]
\end{lemma}
\begin{proof}
Let $\varphi:=\lambda\|\cdot\|_{p,q}^{q}$ and then $F=f+\varphi$. Since $\varphi$ is grouped separable,
by \cite[Proposition 10.5]{Roc98}, it follows from the definition that $d\varphi(0)=\delta_{\{0\}}$ (where $\delta_X$ is the indicator function of $X$). Applying \cite[Proposition 10.9]{Roc98}, it follows that
\begin{equation}\label{boluo1}
dF(0)\ge df(0)+\delta_{\{0\}}.
\end{equation}
Since $f$ is finite and $df(0)(0)=0$,  its subderivative  $df(0)$ is proper (cf. \cite[Exercise 3.19]{Roc98}). Since further $df(0)(0)=0$, it yields that
$df(0)+\delta_{\{0\}}=\delta_{\{0\}}$.
Thus by \eqref{boluo1}, we obtain that $dF(0)\ge \delta_{\{0\}}$.   Therefore, by definition, there exist some $\epsilon>0$ and $\delta>0$ such that
\[
F(x)\ge F(0)+\epsilon\|x\|_2\quad\forall x\in \mathbf{B}(0,\delta).
  \]
The proof is complete.
\end{proof}

With the help of the above lemma, we can present in the following a local version of the  recovery bound. This is done by constructing a path of local minima depending on the regularization parameter $\lambda$ for the regularization problem, which starts from a  sparse solution of the original problem and shares the same support as this sparse solution has,  resulting in a sharper bound in terms of $\lambda^2$.

\begin{theorem}\label{thm-RB-local}
Let $\bar{x}$ be a solution of $Ax=b$, $S$ be the group sparsity of $\bar{x}$, and $B$ be a submatrix of $A$ consisting of its columns corresponding to the active components of $\bar{x}$.
Suppose that any nonzero group of $\bar{x}$ is active,
and that the columns of $A$ corresponding to the active components of $\bar{x}$ are linearly independent.
Let $0<q<1\le p$. Then there exist $\kappa>0$ and a path of local minima of problem \eqref{eq-GSOP}, $x^*(\lambda)$, such that
\[
\|x^*(\lambda)-\bar{x}\|_2^2\le \lambda^2 q^2 S \|(B^\top B)^{-1}\|^2\max_{\bar{x}_{\mathcal{G}_i}\neq 0} \left(\|\bar{x}_{\mathcal{G}_i}\|_p^{2(q-p)}\|\bar{x}_{\mathcal{G}_i}\|_{2p-2}^{2p-2}\right)\quad \forall \lambda<\kappa.
\]
\end{theorem}
\begin{proof}
Without loss of generality, we let $\bar{x}$ be of structure $\bar{x} = (\bar{z}^\top,0)^\top$ with
\begin{equation*}
\mbox{$\bar{z}=(\bar{x}_{\mathcal{G}_1}^\top,\dots,\bar{x}_{\mathcal{G}_S}^\top )^\top$ and  $\bar{x}_{\mathcal{G}_i} \neq_\mathbf{a} 0$ for $i=1,\dots,S$.}
\end{equation*}
and let $s$ be the sparsity of $\bar{x}$.
Let $A=(B,D)$ with $B$ being the submatrix involving the first $s$ columns of $A$ (corresponding to the active components of $\bar{x}$). By the assumption, we have that $B$ is of full column rank and thus $B^\top B$ is invertible.
In this setting, the linear relation $A\bar{x}=b$ reduces to $B\bar{z}=b$. The proof of this theorem is divided into the three steps:
\begin{enumerate}[(a)]
  \item construct a smooth path from $\bar{x}$ by the implicit function theorem;
  \item validate that every point of the constructed path is a local minimum of \eqref{eq-GSOP}; and
  \item establish the recovery bound for the constructed path.
\end{enumerate}
First, to show (a), we define $H:\R^{s+1}\rightarrow \R^{s}$ by
\[
H(z,\lambda)=2 B^\top (Bz-b)+\lambda q\left(\begin{matrix}
   \|z_{\mathcal{G}_1}\|_p^{q-p}\sigma(z_{\mathcal{G}_1})\\\vdots\\ \|z_{\mathcal{G}_S}\|_p^{q-p}\sigma(z_{\mathcal{G}_S})
\end{matrix}\right),
\]
where $\sigma(z_{\mathcal{G}_i})={\rm vector}\left(|z_j|^{p-1}{\rm sign}(z_j)\right)_{\mathcal{G}_i}$, denoting
a vector consisting of $|z_j|^{p-1}{\rm sign}(z_j)$ for all $j\in \mathcal{G}_i$.
Let $\bar{\delta}>0$ be sufficiently small such that  ${\rm sign}(z)={\rm sign}(\bar{z})$ for each $z\in \mathbf{B}(\bar{z},\bar{\delta})$
and thus $H$ is smooth on $\mathbf{B}(\bar{z},\bar{\delta})\times \R$.
Note that $H(\bar{z},0)=0$ and $\frac{\partial H}{\partial z}(\bar{z},0)=2B^\top B$. By the implicit function theorem \cite{Rudin76}, there  exist some $\kappa>0$, $\delta\in (0,\bar{\delta})$ and a unique  smooth function $\xi:(-\kappa, \kappa)\rightarrow \mathbf{B}(\bar{z},\delta)$ such that
\begin{equation}\label{eq-local-1}
\{(z,\lambda)\in \mathbf{B}(\bar{z},\bar{\delta})\times (-\kappa, \kappa): H(z,\lambda)=0\}=\{(\xi(\lambda),\lambda): \lambda\in (-\kappa, \kappa)\},
\end{equation}
and
\begin{equation}\label{eq-local-2}
\frac{d\xi}{d\lambda}=-q\left(2 B^\top B+\lambda q \left(
\begin{array}{cccc}  M_1 &0 & 0\\ 0&\ddots &0 \\ 0&0 & M_\mathcal{S}\\
\end{array}\right)\right)^{-1} \left(\begin{matrix}
   \|\xi(\lambda)_{\mathcal{G}_1}\|_p^{q-p}\sigma(\xi(\lambda)_{\mathcal{G}_1})\\\vdots\\ \|\xi(\lambda)_{\mathcal{G}_S}\|_p^{q-p}\sigma(\xi(\lambda)_{\mathcal{G}_S})
\end{matrix}\right),
\end{equation}
where $M_i$ for each $i=1,\dots,S$ is denoted by
\begin{equation*}
M_i=(q-p)\|\xi(\lambda)_{\mathcal{G}_i}\|_p^{q-2p}(\sigma(\xi(\lambda)_{\mathcal{G}_i})) (\sigma(\xi(\lambda)_{\mathcal{G}_i}))^\top
+(p-1)\|\xi(\lambda)_{\mathcal{G}_i}\|_p^{q-p}{\rm diag}\left(|\xi(\lambda)_j|^{p-2}\right),
\end{equation*}
and ${\rm diag}\left(|\xi(\lambda)_j|^{p-2}\right)$ is a diagonal matrix generated by ${\rm vector}\left(|\xi(\lambda)_j|^{p-2}\right)$.
Thus, due to \eqref{eq-local-1} and \eqref{eq-local-2},
we have constructed a smooth path $\xi(\lambda)$ near $\bar{z}$, $\lambda\in (-\kappa, \kappa)$, be such that
\begin{equation}\label{eq-local-3}
2 B^\top (B\xi(\lambda)-b)+\lambda q\left(\begin{matrix}
\|\xi(\lambda)_{\mathcal{G}_1}\|_p^{q-p}\sigma(\xi(\lambda)_{\mathcal{G}_1})\\\vdots\\ \|\xi(\lambda)_{\mathcal{G}_S}\|_p^{q-p}
\sigma(\xi(\lambda)_{\mathcal{G}_S})
\end{matrix}\right)=0,
\end{equation}
and that
\begin{equation}\label{eq-local-4}
2 B^\top B+\lambda q \left(
\begin{array}{cccc}  M_1 &0 & 0\\ 0&\ddots &0 \\ 0&0 & M_\mathcal{S}\\
\end{array}\right)\succ 0.
\end{equation}

For fixed $\lambda \in (-\kappa,\kappa)$, let $x^*(\lambda):=(\xi(\lambda)^\top,0)^\top$.
To verify (b), we prove that such $x^*(\lambda)$, with $\xi(\lambda)$ satisfying \eqref{eq-local-3} and \eqref{eq-local-4}, is a local minimum of problem \eqref{eq-GSOP}.
Let $h(z):=\|Bz-b\|_2^2+\lambda\|z\|_{p,q}^q$. Note that $h(\xi(\lambda))=\|Ax^*(\lambda)-b\|_2^2+\lambda\|x^*(\lambda)\|_{p,q}^q$ and that $h$ is smooth around $\xi(\lambda)$. By noting that $\xi(\lambda)$ satisfies \eqref{eq-local-3} and \eqref{eq-local-4} (the first- and second- derivative of $h$ at $\xi(\lambda)$), one has that $h$ satisfies the second-order growth condition at $\xi(\lambda)$, that is, there exist some $\epsilon_\lambda >0$ and $\delta_\lambda >0$ such that
\begin{equation}\label{eq-local-0a}
h(z)\geq h(\xi(\lambda))+2\epsilon_\lambda\|z-\xi(\lambda)\|_2^2\quad \forall z\in \mathbf{B}(\xi(\lambda),\delta_\lambda).
\end{equation}
In what follows, let $\epsilon_\lambda >0$ and $\delta_\lambda >0$ be given as above, and select $\epsilon_0>0$ such that
\begin{equation}\label{eq-local-0b}
\sqrt{\epsilon_\lambda \epsilon_0}-\|B\|\|D\|>0.
\end{equation}
According to Lemma \ref{meng-h-y} (with $\|D\cdot\|_2^2+2\langle B\xi(\lambda)-b,  D\cdot \rangle-2\epsilon_0\|\cdot\|_2^2$ in place of $f$),  there exists some $\delta_0>0$ such that
\begin{equation}\label{eq-local-0c}
\|Dy\|_2^2+2\langle B\xi(\lambda)-b,  Dy \rangle-2\epsilon_0\|y\|_2^2+\lambda\|y\|_{p,q}^p\geq 0\quad\forall y\in \mathbf{B}(0,\delta_0).
\end{equation}
Thus, for each $x:=(z,y)\in \mathbf{B}(\xi(\lambda),\delta_\lambda)\times \mathbf{B}(0,\delta_0)$,  it  follows that
\begin{equation*}\label{eq-local-0d}
\begin{array}{ll}
\|Ax-b\|_2^2+\lambda\|x\|_{p,q}^q
&=\|Bz-b+Dy\|_2^2+\lambda\|z\|_{p,q}^q+\lambda\|y\|_{p,q}^q\\
&=\|Bz-b\|_2^2+\lambda\|z\|_{p,q}^q+\|Dy\|_2^2+2\langle Bz-b,  Dy\rangle+\lambda\|y\|_{p,q}^q\\
&=h(z)+\|Dy\|_2^2+2\langle B\xi(\lambda)-b,  Dy \rangle+\lambda\|y\|_{p,q}^q+2\langle B(z-\xi(\lambda)),  Dy\rangle.
\end{array}
\end{equation*}
By \eqref{eq-local-0a} and \eqref{eq-local-0c}, it yields that
\[
\begin{array}{ll}
\|Ax-b\|_2^2+\lambda\|x\|_{p,q}^q
&\ge h(\xi(\lambda))+2\epsilon_\lambda\|z-\xi(\lambda)\|_2^2+2\epsilon_0\|y\|_2^2+2\langle B(z-\xi(\lambda)),  Dy\rangle\\
&\ge h(\xi(\lambda))+4\sqrt{\epsilon_\lambda \epsilon_0}\|z-\xi(\lambda)\|_2\|y\|_2-2\|B\| \|D\| \|z-\xi(\lambda)\|_2  \|y\|_2\\
&= \|Ax^*(\lambda)-b\|_2^2+\lambda\|x^*(\lambda)\|_{p,q}^q+2(2\sqrt{\epsilon_\lambda \epsilon_0}-\|B\| \|D\|)\|z-\xi(\lambda)\|_2\|y\|_2\\
&\geq \|Ax^*(\lambda)-b\|_2^2+\lambda\|x^*(\lambda)\|_{p,q}^q,
\end{array}
\]
where the last inequality follows from \eqref{eq-local-0b}.
Hence $x^*(\lambda)$ is a local minimum of problem \eqref{eq-GSOP}, and (b) is verified.

Finally, we check (c) by providing an upper bound on the distance from $\xi(\lambda)$ to $\bar{z}$.
By \eqref{eq-local-3}, one has that
\begin{equation}\label{eq-local-3a}
\xi(\lambda)-\bar{z}=-\frac{\lambda q}{2} ((B^\top B)^{-1})
\left(\begin{matrix}
\|\xi(\lambda)_{\mathcal{G}_1}\|_p^{q-p}\sigma(\xi(\lambda)_{\mathcal{G}_1})\\\vdots\\ \|\xi(\lambda)_{\mathcal{G}_S}\|_p^{q-p}\sigma(\xi(\lambda)_{\mathcal{G}_S})
\end{matrix}\right).
\end{equation}
As $\{\xi(\lambda):\lambda\in (-\kappa,\kappa)\}\subseteq \mathbf{B}(\bar{z},\bar{\delta})$, without loss of generality,
we assume for each $\lambda< \kappa$ that
\[
\|\xi(\lambda)_{\mathcal{G}_i}\|_p^{2(q-p)} \le 2\|\bar{z}_{\mathcal{G}_i}\|_p^{2(q-p)}\quad
{\rm and}\quad \|\xi(\lambda)_{\mathcal{G}_i}\|_{2p-2}^{2p-2}\le 2\|\bar{z}_{\mathcal{G}_i}\|_{2p-2}^{2p-2}
\quad \forall i=1,\dots,S
\]
(otherwise, we can choose a smaller $\bar{\delta}$).
Recall $\sigma(\xi(\lambda)_{\mathcal{G}_i})={\rm vector}\left(|\xi(\lambda)_j|^{p-1}{\rm sign}(\xi(\lambda)_j)\right)_{\mathcal{G}_i}$. We obtain from \eqref{eq-local-3a} that
\begin{equation*}\label{eq-local-6}
\begin{array}{lll}
\|\xi(\lambda)-\bar{z}\|_2^2&\le \frac{\lambda^2 q^2}{4} \|(B^\top B)^{-1}\|^2 \sum_{i=1}^S
\left(\|\xi(\lambda)_{\mathcal{G}_i}\|_p^{2(q-p)}\sum_{j\in \mathcal{G}_i} |\xi(\lambda)_j|^{2p-2}\right)\\
&= \frac{\lambda^2 q^2}{4} \|(B^\top B)^{-1}\|^2 \sum_{i=1}^S
\left(\|\xi(\lambda)_{\mathcal{G}_i}\|_p^{2(q-p)}\|\xi(\lambda)_{\mathcal{G}_i}\|_{2p-2}^{2p-2}\right)\\
&\le \frac{\lambda^2 q^2}{4} \|(B^\top B)^{-1}\|^2 S \max\limits_{i=1,\dots,S} \left(\|\xi(\lambda)_{\mathcal{G}_i}\|_p^{2(q-p)}\|\xi(\lambda)_{\mathcal{G}_i}\|_{2p-2}^{2p-2}\right)\\
&\le \lambda^2q^2 S \|(B^\top B)^{-1}\|^2\max\limits_{i=1,\dots,S} \left(\|\bar{z}_{\mathcal{G}_i}\|_p^{2(q-p)}\|\bar{z}_{\mathcal{G}_i}\|_{2p-2}^{2p-2}\right).
\end{array}
\end{equation*}
Hence we arrive at that for each $\lambda< \kappa$
\[\|x^*(\lambda)-\bar{x}\|_2^2=\|\xi(\lambda)-\bar{z}\|_2^2\le \lambda^2q^2 S \|(B^\top B)^{-1}\|^2\max_{\bar{x}_{\mathcal{G}_i}\neq 0} \left(\|\bar{x}_{\mathcal{G}_i}\|_p^{2(q-p)}\|\bar{x}_{\mathcal{G}_i}\|_{2p-2}^{2p-2}\right),\]
and the proof is complete.
\end{proof}

Theorem \ref{thm-RB-local} is a significant result in that it provides the uniform local recovery bound for all the $\ell_{p,q}$ regularization problems ($0<q<1$), which is
\[
\|x^*_{p,q}(\lambda)-\bar{x}\|_2^2\le O\left(\lambda^2S\right),
\]
where $x^*_{p,q}(\lambda)$ is a local optimal solution of \eqref{eq-GSOP} (near $\bar{x}$).
This bound improves the global recovery bound given in Theorem \ref{thm-RecoverBound} (of order $\lambda^{\frac{2}{2-q}}$) and
shares the same one with the $\ell_{p,1}$ regularization problem (group Lasso); see Corollary \ref{coro-RB-1}. It is worth noting that when $q=1$ our proof technique is not working as Lemma \ref{meng-h-y} is false in this case.

\section{Proximal gradient method for group sparse optimization}

Many efficient algorithms have been proposed to solve the sparse optimization problem,
and one of the most popular and practical algorithms is the proximal gradient method; see \cite{BeckTeboulle09,Combettes05,Nesterov13,TZhang13} and references therein. It was reported in \cite{BeckTeboulle09,Combettes05,Nesterov13} that the proximal gradient method for solving the $\ell_1$ regularization problem reduces to the well-known iterative soft thresholding algorithm and that the accelerated proximal gradient methods  proposed in \cite{BeckTeboulle09,Nesterov13} have the convergence rate $O(1/k^2)$. Recently, the convergence theory of the proximal gradient method for solving the nonconvex regularization problem was studied under the framework of the Kurdyka-{\L}ojasiewicz theory \cite{AttouchBolte10,BolteTeboulle13},  the majorization-minimization (MM) scheme \cite{Mairal13}, the coordinate gradient descent method \cite{TsengPaul09} and the successive upper-bound minimization approach \cite{RazaviyaynHong13}.

In this section, we apply the proximal gradient method
to solve the group sparse optimization (PGM-GSO) via $\ell_{p,q}$ ($p\ge 1,q\ge 0$) regularization \eqref{eq-GSOP}, which is stated as follows.

\noindent \vskip 0.3cm

\N \textbf{PGM-GSO}\\
Select a stepsize $v$, start with an initial point $x_0\in \R^n$, and generate a sequence $\{x^k\}\subseteq \R^n$ via the iteration
\begin{eqnarray}
z^k&=&x^k-2v A^\top  (Ax^k-b),  \\
x^{k+1}&\in&{\rm Arg}\min_{x\in \R^n}\left\{\lambda \|x\|_{p,q}^q+\frac{1}{2v}\|x-z^k\|_2^2\right\}. \label{eq-PPA-Lp,q}
\end{eqnarray}


We will obtain the analytical solutions of \eqref{eq-PPA-Lp,q} for some specific $p$ and $q$, and the linear convergence rate of the PGM-GSO.
The convergence theory of the PGM-GSO falls in the framework of the Kurdyka-{\L}ojasiewicz theory; see \cite{AttouchBolte10,BolteTeboulle13}. In particular, following from \cite[Theorem 1 and  Proposition 3]{BolteTeboulle13}, the sequence generated by the PGM-GSO converges to a critical point, especially a global minimum when $q\ge 1$ and  a local minimum when $q=0$ (inspired by the idea in \cite{Blumensath08}), as summarized as follows.

\begin{theorem}\label{thm-GP-SO-2}
Let $p\ge 1$. Suppose that the sequence $\{x^k\}$ is generated by the {\rm PGM-GSO} with
$v< \frac{1}{2}\|A\|_2^{-2}$. Then the following statements hold:
\begin{enumerate}[{\rm (i)}]
\item if $q\ge 1$, then $\{x^k\}$ converges to a global minimum of problem \eqref{eq-GSOP},
\item if $q=0$, then $\{x^k\}$ converges to a local minimum of problem \eqref{eq-GSOP}, and
\item if $0<q<1$, then $\{x^k\}$ converges to a critical point\footnote{A point $x$ is said to be a critical point of $F$ if 0 belongs to its limiting subdifferential at $x$ \cite{Mordukhovich-VA}.} of problem \eqref{eq-GSOP}.
\end{enumerate}
\end{theorem}

\subsection{Analytical solutions of \eqref{eq-PPA-Lp,q}}

Since the main computation of the proximal gradient method is the proximal step \eqref{eq-PPA-Lp,q},
it is significant to investigate the solutions of \eqref{eq-PPA-Lp,q} for the specific applications.
Note that $\|x\|_{p,q}^q$ and $\|x-z^k\|_2^2$ are both grouped separable.
Then the proximal step \eqref{eq-PPA-Lp,q} can be achieved parallelly in each group, and is equivalent to solve
a cycle of low dimensional proximal optimization subproblems
\begin{equation}\label{eq-PPA-Lp,q-subproblem}
x^{k+1}_{\mathcal{G}_i}\in{\rm Arg}\min_{x\in \R^{n_i}}\{\lambda  \|x_{\mathcal{G}_i}\|_p^q+\frac{1}{2v}\|x_{\mathcal{G}_i}-z^k_{\mathcal{G}_i}\|_2^2\}, ~{\rm for}~ i=1,\cdots,r.
\end{equation}
When $p$ and $q$ are given as some specific numbers, such as $p=1,2$ and $q=0,1/2,2/3,1$, the solution of subproblem \eqref{eq-PPA-Lp,q-subproblem} of each group can be given explicitly by an analytical formula, as shown in the following proposition.
\begin{proposition}\label{thm-Lp,q-formula}
Let $z\in \R^l$, $v> 0$ and the proximal regularization be $Q_{p,q}(x):=\lambda \|x\|_p^q+\frac{1}{2v}\|x-z\|_2^2$. Then the proximal operator
\begin{equation*}\label{eq-GSOP-p,q-proximal}
P_{p,q}(z)\in {\rm Arg}\min_{x\in \R^l}\left\{Q_{p,q}(x)\right\}
\end{equation*}
has the following analytical formula:
\begin{enumerate}[{\rm (i)}]
  \item if $p=2$ and $q=1$, then
  \begin{equation}\label{eq-GSOP-2,1-solution}
P_{2,1}(z)=\left\{\begin{matrix}
   \left(1-\frac{v\lambda}{\|z\|_2}\right)z,&{\|z\|_2>v\lambda,}\\
   0, &{\rm otherwise,}
\end{matrix}\right.
\end{equation}
  \item if $p=2$ and $q=0$, then
  \begin{equation}\label{eq-GSOP-2,0-solution}
P_{p,0}(z)=\left\{\begin{matrix}
   z,&{\|z\|_2>\sqrt{2v\lambda},}\\
   0~ {\rm or}~ z, &{\|z\|_2=\sqrt{2v\lambda},}\\
   0, &{\|z\|_2<\sqrt{2v\lambda},}
\end{matrix}\right.
\end{equation}
  \item if $p=2$ and $q=1/2$, then
  \begin{equation}\label{eq-GSOP-2,1/2-solution}
P_{2,1/2}(z)=\left\{\begin{matrix}
   \frac{16 \|z\|_2^{3/2} \cos^3\left(\frac{\pi}{3}-\frac{\psi(z)}{3}\right)}{3\sqrt{3}v\lambda +16 \|z\|_2^{3/2} \cos^3\left(\frac{\pi}{3}-\frac{\psi(z)}{3}\right)}z,&{\|z\|_2>\frac{3}{2}(v\lambda)^{2/3},}\\
   0~{\rm or}~\frac{16 \|z\|_2^{3/2} \cos^3\left(\frac{\pi}{3}-\frac{\psi(z)}{3}\right)}{3\sqrt{3}v\lambda +16 \|z\|_2^{3/2} \cos^3\left(\frac{\pi}{3}-\frac{\psi(z)}{3}\right)}z,&{\|z\|_2=\frac{3}{2}(v\lambda)^{2/3},}\\
   0, &{\|z\|_2<\frac{3}{2}(v\lambda)^{2/3},}
\end{matrix}\right.
\end{equation}
with
\begin{equation}\label{eq-GSOP-2,1/2-solution-2}
\psi(z)=\arccos \left(\frac{v\lambda}{4}\left(\frac{3}{\|z\|_2}\right)^{3/2}\right),
\end{equation}
  \item if $p=1$ and $q=1/2$, then
  \begin{equation}\label{eq-GSOP-1,1/2-solution}
P_{1,1/2}(z)=\left\{\begin{matrix}
   \tilde{z},&{Q_{1,1/2}(\tilde{z})<Q_{1,1/2}(0),}\\
   0~{\rm or}~\tilde{z},&{Q_{1,1/2}(\tilde{z})=Q_{1,1/2}(0),}\\
   0, &{Q_{1,1/2}(\tilde{z})>Q_{1,1/2}(0),}
\end{matrix}\right.
\end{equation}
with
\begin{equation*}\label{eq-GSOP-1,1/2-solution-2}
\tilde{z}=z-\frac{\sqrt{3}v\lambda}{4 \sqrt{\|z\|_1} \cos\left(\frac{\pi}{3}-\frac{\xi(z)}{3}\right)}{\rm sign}(z),\quad \xi(z)=\arccos \left(\frac{v\lambda l}{4}\left(\frac{3}{\|z\|_1}\right)^{3/2}\right),
\end{equation*}
  \item if $p=2$ and $q=2/3$, then
  \begin{equation}\label{eq-GSOP-2,2/3-solution}
P_{2,2/3}(z)=\left\{\begin{matrix}
   \frac{3\left(a^{3/2}+\sqrt{2\|z\|_2-a^3}\right)}{32v\lambda a^2+3\left(a^{3/2}+\sqrt{2\|z\|_2-a^3}\right)}z,&{\|z\|_2>2\left(\frac23v\lambda\right)^{3/4},}\\
   0~{\rm or}~\frac{3\left(a^{3/2}+\sqrt{2\|z\|_2-a^3}\right)}{32v\lambda a^2+3\left(a^{3/2}+\sqrt{2\|z\|_2-a^3}\right)}z,&{\|z\|_2=2\left(\frac23v\lambda\right)^{3/4},}\\
   0, &{\|z\|_2<2\left(\frac23v\lambda\right)^{3/4},}
\end{matrix}\right.
\end{equation}
with
\begin{equation}\label{eq-GSOP-2,2/3-solution-2}
a=\frac{2}{\sqrt{3}}(2v\lambda)^{1/4}\left(\cosh\left(\frac{\varphi(z)}{3}\right)\right)^{1/2},\quad \varphi(z)={\rm arccosh} \left( \frac{27\|z\|_2^2}{16(2v\lambda)^{3/2}}\right),
\end{equation}
  \item if $p=1$ and $q=2/3$, then
  \begin{equation}\label{eq-GSOP-1,2/3-solution}
P_{1,2/3}(z)=\left\{\begin{matrix}
   \bar{z},&{Q_{1,2/3}(\bar{z})<Q_{1,2/3}(0),}\\
   0~{\rm or}~\bar{z},&{Q_{1,2/3}(\bar{z})=Q_{1,2/3}(0),}\\
   0, &{Q_{1,2/3}(\bar{z})>Q_{1,2/3}(0),}
\end{matrix}\right.
\end{equation}
with
\begin{equation*}\label{eq-GSOP-1,2/3-solution-2}
\bar{z}=z-\frac{4v\lambda\bar{a}^{1/2}}{3\left(\bar{a}^{3/2}+\sqrt{2\|z\|_1-\bar{a}^3}\right)}{\rm sign}(z),
\end{equation*}
and
\begin{equation*}\label{eq-GSOP-1,2/3-solution-3}
\bar{a}=\frac{2}{\sqrt{3}}(2v\lambda l)^{1/4}\left(\cosh\left(\frac{\zeta(z)}{3}\right)\right)^{1/2},
\quad \zeta(z)={\rm arccosh} \left( \frac{27\|z\|_1^2}{16(2v\lambda l)^{3/2}}\right).
\end{equation*}
\end{enumerate}
\end{proposition}
\begin{proof}
Since the proximal regularization $Q_{p,q}(x):=\lambda \|x\|_p^q+\frac{1}{2v}\|x-z\|_2^2$ is non-differentiable only at $0$, $P_{p,q}(z)$ must be 0 or some point $\tilde{x}(\neq 0)$ satisfying the first-order condition
\begin{equation}\label{eq-FOC}
\lambda q \|\tilde{x}\|_p^{q-p}\left(\begin{matrix}
|\tilde{x}_1|^{p-1} {\rm sign}(\tilde{x}_1)\\\vdots\\ |\tilde{x}_l|^{p-1} {\rm sign}(\tilde{x}_l)
\end{matrix}\right)+\frac{1}{v}(\tilde{x}-z)=0.
\end{equation}
Thus, to derive the analytical formula of the proximal operator $P_{p,q}(z)$, we just need to calculate such $\tilde{x}$ via \eqref{eq-FOC}, and then compare the objective function values $Q_{p,q}(\tilde{x})$ and $Q_{p,q}(0)$ to obtain the solution inducing the smaller value.
The proofs of the six statements follow in the above routine, and we only provide the detailed proofs of (iii) and (v) as samples.
\begin{enumerate}
  \item[(iii)] When $p=2$ and $q=1/2$, \eqref{eq-FOC} reduces to
\begin{equation}\label{eq-optimalityPO-2,1/2}
\frac{\lambda \tilde{x}}{2\|\tilde{x}\|_2^{3/2}}+\frac{1}{v}(\tilde{x}-z)=0,
\end{equation}
and consequently,
\begin{equation}\label{eq-optimalityPO2-2,1/2}
\|\tilde{x}\|_2^{3/2}-\|z\|_2\|\tilde{x}\|_2^{1/2}+\frac{1}{2}v\lambda=0.
\end{equation}
Denote $\eta=\|\tilde{x}\|_2^{1/2}> 0$. The equation \eqref{eq-optimalityPO2-2,1/2} can be transformed into the following cubic algebraic equation
\begin{equation}\label{eq-GSOP-2,1/2-CubicAE}
\eta^3-\|z\|_2\eta+\frac{1}{2}v\lambda=0.
\end{equation}
Due to the hyperbolic solution of the cubic equation (see \cite{Short37}), by denoting
\[
r=2\sqrt{\frac{\|z\|_2}{3}},~~\alpha=\arccos \left(\frac{v\lambda}{4}\left(\frac{3}{\|z\|_2}\right)^{3/2}\right)(:=\psi(z))~~{\rm and}~~\beta={\rm arccosh}\left(-\frac{v\lambda}{4}\left(\frac{3}{\|z\|_2}\right)^{3/2}\right),
\]
the solution of \eqref{eq-GSOP-2,1/2-CubicAE} can be expressed as the follows.
\begin{enumerate}[(1)]
\item If $0\le \|z\|_2\le 3\left(\frac{v\lambda}{4}\right)^{2/3}$, then the three roots of \eqref{eq-GSOP-2,1/2-CubicAE} are given by
\begin{equation*}
\eta_1=r\cosh \frac{\beta}{3},~\eta_2=-\frac{r}{2}\cosh \frac{\beta}{3}+i \frac{\sqrt{3}r}{2}\sinh \frac{\beta}{3},~
\eta_3=-\frac{r}{2}\cosh \frac{\beta}{3}-i \frac{\sqrt{3}r}{2}\sinh \frac{\beta}{3},
\end{equation*}
where $i$ denotes the imaginary unit.
However, this $\beta$ does not exist since the value of hyperbolic cosine must be positive.
Thus, in this case, $P_{2,1/2}(z)=0$.
\item If $\|z\|_2> 3\left(\frac{v\lambda}{4}\right)^{2/3}$, then the three roots of \eqref{eq-GSOP-2,1/2-CubicAE} are
\begin{equation*}
\eta_1=r\cos \left(\frac{\pi}{3}-\frac{\alpha}{3}\right),~\eta_2=-r\sin \left(\frac{\pi}{2}-\frac{\alpha}{3}\right),~
\eta_3=-r\cos \left(\frac{2\pi}{3}-\frac{\alpha}{3}\right).
\end{equation*}
The unique positive solution of \eqref{eq-GSOP-2,1/2-CubicAE} is $\|\tilde{x}\|_2^{1/2}=\eta_1$, and thus, the unique solution  of \eqref{eq-optimalityPO-2,1/2} is given by
\begin{equation*}
\tilde{x}=\frac{2 \eta_1^3}{v\lambda +2\eta_1^3}z
=\frac{16 \|z\|_2^{3/2} \cos^3\left(\frac{\pi}{3}-\frac{\psi(z)}{3}\right)}{3\sqrt{3}v\lambda +16 \|z\|_2^{3/2} \cos^3\left(\frac{\pi}{3}-\frac{\psi(z)}{3}\right)}z.
\end{equation*}
\end{enumerate}
Finally, we compare the objective function values $Q_{2,1/2}(\tilde{x})$ and $Q_{2,1/2}(0)$.
For this purpose, when $\|z\|_2>  3\left(\frac{v\lambda}{4}\right)^{2/3}$, we define
\[
\begin{array}{lll}
H(\|z\|_2)&:=\frac{v}{\|\tilde{x}\|_2}\left(Q_{2,1/2}(0)-Q_{2,1/2}(\tilde{x})\right)\\
&=\frac{v}{\|\tilde{x}\|_2}\left(\frac{1}{2v}\|z\|_2^2-\lambda\|\tilde{x}\|_2^{1/2}-\frac{1}{2v}\|\tilde{x}-z\|_2^2\right)\\
&=\|z\|_2-\frac{\|\tilde{x}\|_2^2+2v\lambda \|\tilde{x}\|_2^{1/2}}{2\|\tilde{x}\|_2}\\
&=\frac{1}{2}\|z\|_2-\frac{3}{4}v\lambda \|\tilde{x}\|_2^{-1/2},
\end{array}
\]
where the third equality holds since that $\tilde{x}$ is proportional to $z$, and fourth equality follows from \eqref{eq-optimalityPO2-2,1/2}.
Since both $\|z\|_2$ and $\|\tilde{x}\|_2$ are strictly increasing on $\|z\|_2$, $H(\|z\|_2)$ is also strictly increasing when
$\|z\|_2>  3\left(\frac{v\lambda}{4}\right)^{2/3}$. Thus the unique solution of $H(\|z\|_2)=0$ satisfies
\[
\|z\|_2 \|\tilde{x}\|_2^{1/2}=\frac{3}{2} v\lambda,
\]
and further, \eqref{eq-optimalityPO2-2,1/2} implies that the solution of $H(\|z\|_2)=0$ is
\[
\|z\|_2=\frac{3}{2}(v\lambda)^{2/3}.
\]
Therefore, we arrive at the formulae \eqref{eq-GSOP-2,1/2-solution} and \eqref{eq-GSOP-2,1/2-solution-2}.
\item[(v)] When $p=2$ and $q=2/3$, \eqref{eq-FOC} reduces to
\begin{equation}\label{eq-optimalityPO-2,2/3}
\frac{2\lambda \tilde{x}}{3\|\tilde{x}\|_2^{4/3}}+\frac{1}{v}(\tilde{x}-z)=0,
\end{equation}
and consequently,
\begin{equation}\label{eq-optimalityPO2-2,2/3}
\|\tilde{x}\|_2^{4/3}-\|z\|_2\|\tilde{x}\|_2^{1/3}+\frac{2}{3}v\lambda=0.
\end{equation}
Denote $\eta=\|\tilde{x}\|_2^{1/3}> 0$ and $h(t)=t^4-\|z\|_2t+\frac{2}{3}v\lambda$. Thus, $\eta$ is the positive solution of $h(t)=0$. Next, we seek $\eta$ by the method of undetermined coefficients. Assume that
\begin{equation}\label{eq-appd-h0}
h(t)=t^4-\|z\|_2t+\frac{2}{3}v\lambda=(t^2+at+b)(t^2+ct+d),\quad {\rm where}~ a,b,c,d\in \R.
\end{equation}
By expansion and comparison, we have
\[a+c=0,\quad b+d+ac=0,\quad ad+bc=-\|z\|_2,\quad bd=\frac{2}{3}v\lambda,\]
and thus,
\begin{equation}\label{eq-appd-2}
c=-a, b=\frac12\left(a^2+\frac{\|z\|_2}{a}\right), d=\frac12\left(a^2-\frac{\|z\|_2}{a}\right), bd=\frac14\left(a^4-\frac{\|z\|_2^2}{a^2}\right)=\frac{2}{3}v\lambda.
\end{equation}
By letting $M=a^2$, the last one of the above equalities reduces to the following cubic algebraic equation
\begin{equation}\label{eq-appd-1}
M^3-\frac{8}{3}v\lambda M-\|z\|_2^2=0.
\end{equation}
According to the Cardano formula for the cubic equation, the root of \eqref{eq-appd-1} can be represented by
\begin{equation*}\label{eq-appd-a}
a^2=M=\left(\frac{\|z\|_2^2}{2}+\sqrt{\frac{\|z\|_2^4}{4}-\left(\frac89 v\lambda\right)^3}\right)^{1/3}
+\left(\frac{\|z\|_2^2}{2}-\sqrt{\frac{\|z\|_2^4}{4}-\left(\frac89 v\lambda\right)^3}\right)^{1/3},
\end{equation*}
which can also be reformulated in the following hyperbolic form (see \cite{Short37})
\begin{equation}\label{eq-appd-b}
a^2=M=\frac43\sqrt{2v\lambda}\cosh\left(\frac{\varphi(z)}{3}\right),
\end{equation}
where $\varphi(z)$ is given by \eqref{eq-GSOP-2,2/3-solution-2}.
By \eqref{eq-appd-h0} and \eqref{eq-appd-2}, we have that $\eta$, the positive root of $h(t)=0$, satisfies
\[\eta^2+a\eta+\frac12\left(a^2+\frac{\|z\|_2}{a}\right)=0\quad {\rm or}\quad \eta^2-a\eta+\frac12\left(a^2-\frac{\|z\|_2}{a}\right)=0.\]
Hence, the real roots of the above equations, that is, the real roots of $h(t)=0$, are
\begin{equation}\label{eq-appd-3}
\eta_1=\frac12\left(|a|+\sqrt{\frac{2\|z\|_2}{|a|}-a^2}\right), \quad \eta_2=\frac12\left(|a|-\sqrt{\frac{2\|z\|_2}{|a|}-a^2}\right).
\end{equation}
It is easy to see that $\eta_1>\eta_2$ and that $\eta_2$ should be discarded as it induces the saddle point rather than a minimum (since $h(t)>0$ when $t<\eta_2$). Thus, by \eqref{eq-optimalityPO-2,2/3}, \eqref{eq-appd-b} and \eqref{eq-appd-3}, one has
\[
\tilde{x}=\frac{3\eta_1^4}{2v\lambda+3\eta_1^4}z=\frac{3\left(a^{3/2}+\sqrt{2\|z\|_2-a^3}\right)}{32v\lambda a^2+3\left(a^{3/2}+\sqrt{2\|z\|_2-a^3}\right)}z,
\]
where $a$ is given by \eqref{eq-GSOP-2,2/3-solution-2}.
Finally, we compare the objective function values $Q_{2,2/3}(\tilde{x})$ and $Q_{2,2/3}(0)$.
For this purpose, we define
\[
\begin{array}{lll}
H(\|z\|_2)&:=\frac{v}{\|\tilde{x}\|_2}\left(Q_{2,2/3}(0)-Q_{2,2/3}(\tilde{x})\right)\\
&=\frac{v}{\|\tilde{x}\|_2}\left(\frac{1}{2v}\|z\|_2^2-\lambda\|\tilde{x}\|_2^{2/3}-\frac{1}{2v}\|\tilde{x}-z\|_2^2\right)\\
&=\|z\|_2-\frac{\|\tilde{x}\|_2^2+2v\lambda \|\tilde{x}\|_2^{2/3}}{2\|\tilde{x}\|_2}\\
&=\frac{1}{2}\|z\|_2-\frac{2}{3}v\lambda \|\tilde{x}\|_2^{-1/3},
\end{array}
\]
where the third equality holds since that $\tilde{x}$ is proportional to $z$, and fourth equality follows from \eqref{eq-optimalityPO2-2,2/3}.
Since both $\|z\|_2$ and $\|\tilde{x}\|_2$ are strictly increasing on $\|z\|_2$, $H(\|z\|_2)$ is also strictly increasing when
$\|z\|_2>  4(\frac29 v\lambda)^{3/4}$. Thus the unique solution of $H(\|z\|_2)=0$ satisfies
\[
\|z\|_2 \|\tilde{x}\|_2^{1/3}=\frac{4}{3} v\lambda,
\]
and further, \eqref{eq-optimalityPO2-2,2/3} implies that the solution of $H(\|z\|_2)=0$ is
\[
\|z\|_2=2\left(\frac23v\lambda\right)^{3/4}.
\]
Therefore, we arrive at the formulae \eqref{eq-GSOP-2,2/3-solution} and \eqref{eq-GSOP-2,2/3-solution-2}.
\end{enumerate}
\end{proof}

\begin{remark}
Note from \eqref{eq-GSOP-2,0-solution}, \eqref{eq-GSOP-2,1/2-solution}, \eqref{eq-GSOP-1,1/2-solution}, \eqref{eq-GSOP-2,2/3-solution} and \eqref{eq-GSOP-1,2/3-solution} that the solutions of the proximal optimization subproblems might not be unique when $Q_{p,q}(\tilde{x})=Q_{p,q}(0)$. To avoid this obstacle in numerical computations,
we select the solution $P_{p,q}(z)=0$ whenever $Q_{p,q}(\tilde{x})=Q_{p,q}(0)$, which achieves a more sparse solution, in the definition of the proximal operator to guarantee a unique update.
\end{remark}

\begin{remark}
By Proposition \ref{thm-Lp,q-formula}, one sees that the proximal gradient method meets the group sparsity structure,
since the components of each iterate within each group are likely to be either all zeros or all nonzeros.
When $n_{\max}=1$, the data do not form any group structure in the feature space,
and the sparsity is achieved only on the individual feature level.
In this case, the proximal operators $P_{2,1}(z)$, $P_{2,0}(z)$, and $P_{2,1/2}(z)$ and $P_{1,1/2}(z)$ reduce to
the soft thresholding function in \cite{Daubechies04}, the hard thresholding function in \cite{Blumensath08} and
the half thresholding function in \cite{XuZB12}, respectively.
\end{remark}


\begin{remark}
Proposition \ref{thm-Lp,q-formula} presents the analytical solution of the proximal optimization subproblems \eqref{eq-PPA-Lp,q-subproblem} when $q=0,1/2,2/3,1$.
However, in other cases, the analytical solution of \eqref{eq-PPA-Lp,q-subproblem} seems not available, since the algebraic equation \eqref{eq-FOC} does
not have an analytical solution (it is difficult to find an analytical solution for the algebraic equation whose order is larger than four).
Thus, in the general cases of $q\in (0,1)$, we alternatively use the Newton method to solve the nonlinear equation \eqref{eq-FOC}, which is the optimality condition of the proximal optimization subproblem.
The numerical simulation in Figure \ref{fig-vary-q} of Section 4 shows that the Newton method works in solving the proximal optimization subproblems \eqref{eq-PPA-Lp,q-subproblem} for the general $q$, while the $\ell_{p,1/2}$ regularization is the best one among the $\ell_{p,q}$ regularizations for $q\in[0,1]$.
\end{remark}

\subsection{Linear convergence rate}
Recall that convergence results of PGM-GSO are from the references \cite{BeckTeboulle09,Blumensath08,BolteTeboulle13}, saying that the generated sequence globally converges to a critical point or a global/local minimum of the $\ell_{p,q}$ regularization problem. However, the result on convergence rates of the proximal gradient method for solving lower-order regularization problems is still undiscovered.
In this subsection, we will establish the linear convergence rate of PGM-GSO for the case $p=1$ and $0<q<1$.

%
%
%
%
By virtue of the second-order necessary condition of \eqref{eq-PPA-Lp,q}, the following lemma provides a lower bound for nonzero groups of sequence $\{x^k\}$ generated by the PGM-GSO and shows that the index set of nonzero groups of $\{x^k\}$ maintains constant for large $k$.

\begin{lemma}\label{lem-LB}
Let $p=1$, $0<q<1$ and $K=\left(v \lambda q (1-q)\right)^{\frac{1}{2-q}}$. Let $\{x^k\}$ be a sequence generated by the {\rm PGM-GSO} with $v< \frac{1}{2}\|A\|_2^{-2}$.
Then the following statements hold:
\begin{enumerate}[{\rm (i)}]
\item for any $i$ and $k$, if $x^k_{\mathcal{G}_i}\neq 0$, then $\|x^k_{\mathcal{G}_i}\|_1\ge K$.
\item $x^k$ shares the same index set of nonzero groups for large $k$, that is, there exist $N\in \mathbb{N}$ and $\mathcal{I}\subseteq \{1,\dots,r\}$ such that
\begin{equation*}
\left\{\begin{matrix}
   x^k_{\mathcal{G}_i}\neq 0,&{i\in \mathcal{I},}\\
   x^k_{\mathcal{G}_i}= 0, &{i\notin \mathcal{I}},
\end{matrix}\right. \quad \mbox{for all $k\ge N$}.
\end{equation*}
\end{enumerate}
\end{lemma}
\begin{proof}
\begin{enumerate}[{\rm (i)}]
\item For each group $x^k_{\mathcal{G}_i}$, by \eqref{eq-PPA-Lp,q}, one has that
\begin{equation}\label{eq-GPA-LB-1}
x^{k}_{\mathcal{G}_i}\in {\rm Arg}\min_{x\in \R^{n_i}}\left\{\lambda \|x\|_1^q+\frac{1}{2v}\|x-z^{k-1}_{\mathcal{G}_i}\|_2^2\right\}.
\end{equation}
If $x^k_{\mathcal{G}_i} \neq 0$, we define $\mathcal{A}_i^k:=\{j\in \mathcal{G}_i: x^k_j\neq 0\}$ and $a_i^k:=|\mathcal{A}_i^k|$. Without loss of generality, we assume that the first $a_i^k$ components of $x^k_{\mathcal{G}_i}$ are nonzeros. Then \eqref{eq-GPA-LB-1} implies that
\begin{equation}\label{eq-GPA-LB-1a}
x^{k}_{\mathcal{G}_i}\in {\rm Arg}\min_{x\in \R^{a_i^k}\times\{0\}}\left\{\lambda \|x\|_1^q+\frac{1}{2v}\|x-z^{k-1}_{\mathcal{G}_i}\|_2^2\right\}.
\end{equation}

The second-order necessary condition of \eqref{eq-GPA-LB-1a} implies that
\[
\frac1v I_i^k + \lambda q (q-1) M_i^k \succeq 0,
\]
where $I_i^k$ is the identity matrix in $R^{a_i^k \times a_i^k}$ and $M_i^k=\|x^{k}_{\mathcal{A}_i^k}\|_1^{q-2} ({\rm sign}(x^{k}_{\mathcal{A}_i^k}))({\rm sign}(x^{k}_{\mathcal{A}_i^k}))^\top$.
Let $e$ be the first column of $I_i^k$. Therefore, we obtain that
\[
\frac1v e^\top I_i^k e + \lambda q (q-1) e^\top M_i^k e \ge 0,
\]
that is,
\[
\frac1v+\lambda q (q-1)\|x^{k}_{\mathcal{A}_i^k}\|_1^{q-2}\ge 0.
\]
Consequently, it implies that
\[
\|x^{k}_{\mathcal{G}_i}\|_1=\|x^{k}_{\mathcal{A}_i^k}\|_1\ge \left(v \lambda q (1-q)\right)^{\frac{1}{2-q}} =K.
\]
Hence, it completes the proof of (i).
\item Recall from Theorem \ref{thm-GP-SO-2} that $\{x^k\}$ converges to a critical point $x^*$. Then there exists $N\in \mathbb{N}$ such that $\|x^k-x^*\|_2< \frac{K}{2\sqrt{n}}$, and thus,
\begin{equation}\label{eq-GPA-LB-2}
\|x^{k+1}-x^k\|_2\le \|x^{k+1}-x^*\|_2+\|x^k-x^*\|_2< \frac{K}{\sqrt{n}},
\end{equation}
 for any $k\ge N$. Proving by contradiction, without loss of generality, we assume that there exist $k\ge N$ and $i\in \{1,\dots,r\}$ such that $x^{k+1}_{\mathcal{G}_i}\neq 0$ and $x^k_{\mathcal{G}_i}= 0$. Then it follows from (i) that
\begin{equation}\label{eq-GPA-LB-3}
\|x^{k+1}-x^k\|_2\ge \frac{1}{\sqrt{n}}\|x^{k+1}-x^k\|_1\ge \frac{1}{\sqrt{n}} \|x^{k+1}_{\mathcal{G}_i}-x^{k}_{\mathcal{G}_i}\|_1\ge \frac{K}{\sqrt{n}}.
\end{equation}
Hence we arrive at a contradiction from \eqref{eq-GPA-LB-2} and \eqref{eq-GPA-LB-3}. The proof is complete.
\end{enumerate}
\end{proof}

The following lemma provides the first- and second-order conditions for a local minimum of $\ell_{1,q}$ regularization problem.

\begin{lemma}\label{thm-SOG}
Let $p=1$ and $0<q<1$.
Assume that $x^*$ is a local  minimum of \eqref{eq-GSOP}, and that any nonzero group of $x^*$ is active; without loss of generality, we assume that $x^*$ is of structure ${x^*} = ({y^*}^\top,0)^\top$ with
\begin{equation}\label{eq-structure}
\mbox{$y^*=({x^*_{\mathcal{G}_1}}^\top,\dots,{x^*_{\mathcal{G}_S}}^\top )^\top$ and  $x^*_{\mathcal{G}_i} \neq_\mathbf{a} 0$ for $i=1,\dots,S$.}
\end{equation}
Let $A=(B,D)$, where $B$ is a submatrix corresponding to $y^*$, i.e., $B=(A_{\cdot j})$ with $j\in \{\mathcal{G}_i:i\in \mathcal{S}\}$ and $D=(A_{\cdot j})$ with $j\in \{\mathcal{G}_i:i\in \mathcal{S}^c\}$. Consider the following restricted problem
 \begin{equation}\label{eq-SOG-3}
\min_{y\in \R^{n_\mathbf{a}}}\quad  f(y)+\varphi(y),
\end{equation}
where $n_{\mathbf{a}}:=\sum_{i\in \mathcal{S}} n_i$, and
\[
f:\R^{n_\mathbf{a}}\to \R \quad {\rm by} \quad f(y):=\|By-b\|_2^2\quad \mbox{for any } y\in \R^{n_\mathbf{a}},
\]
\[
\varphi:\R^{n_\mathbf{a}}\to \R \quad {\rm by} \quad \varphi(y):=\lambda \|y\|_{1,q}^q \quad \mbox{for any } y\in \R^{n_\mathbf{a}}.
\]
Then the following statements are true:
\begin{enumerate}[{\rm (i)}]
  \item The following first- and second-order conditions hold
  \begin{equation}\label{eq-SOG-1}
2B^\top (By^*-b)+\lambda q
\left( \begin{array}{c}  \|y^*_{\mathcal{G}_1}\|_1^{q-1}{\rm sign}(y^*_{\mathcal{G}_1})\\ \vdots \\
\|y^*_{\mathcal{G}_S}\|_1^{q-1}{\rm sign}(y^*_{\mathcal{G}_S})\\
\end{array}\right)=0,
\end{equation}
and
\begin{equation}\label{eq-SOG-2}
2B^\top B+\lambda q(q-1)\left(\begin{array}{cccc}  M^*_1 &0 & 0\\ 0&\ddots &0 \\ 0&0 & M^*_S\\
\end{array}\right)\succ 0,
\end{equation}
where $M^*_i=\|y^*_{\mathcal{G}_i}\|_1^{q-2}\left({\rm sign}(y^*_{\mathcal{G}_i})\right)\left({\rm sign}(y^*_{\mathcal{G}_i})\right)^\top$.
  \item The second-order growth condition holds at $y^*$ for problem \eqref{eq-SOG-3},
that is, there exist $\varepsilon>0$ and $\delta>0$ such that
\begin{equation}\label{eq-SOG-4}
(f+\varphi)(y)\geq (f+\varphi)(y^*)+\varepsilon\|y-y^*\|_2^2\quad \mbox{for any } y\in B(y^*,\delta).
\end{equation}
\end{enumerate}
 \end{lemma}
\begin{proof}
\begin{enumerate}[(i)]
  \item
By \eqref{eq-structure}, one has that $\varphi(\cdot)$ is smooth around $y^*$ with its first- and second-derivatives being
\[
\varphi'(y^*) = \lambda q
\left( \begin{array}{c}  \|y^*_{\mathcal{G}_1}\|_1^{q-1}{\rm sign}(y^*_{\mathcal{G}_1})\\ \vdots \\
\|y^*_{\mathcal{G}_S}\|_1^{q-1}{\rm sign}(y^*_{\mathcal{G}_S})\\
\end{array}\right),
\]
and
\[
\varphi''(y^*) = \lambda q(q-1)\left(\begin{array}{cccc}  M^*_1 &0 & 0\\ 0&\ddots &0 \\ 0&0 & M^*_S\\
\end{array}\right);
\]
hence $(f+\varphi)(\cdot)$ is also smooth around $y^*$.
   Therefore we obtain the following first- and second-order necessary conditions of \eqref{eq-SOG-3}
   \[
   f'(y^*)+\varphi'(y^*)=0\quad {\rm and}\quad f''(y^*)+\varphi''(y^*)\succeq 0,
   \]
   which are \eqref{eq-SOG-1} and
   \begin{equation}\label{eq-SOG-2a}
    2B^\top B+\lambda q(q-1)\left(\begin{array}{cccc}  M^*_1 &0 & 0\\ 0&\ddots &0 \\ 0&0 & M^*_S\\
    \end{array}\right)\succeq 0,
    \end{equation}
   respectively.
   Proving by contradiction, we assume that (\ref{eq-SOG-2}) does not hold, i.e., there exists some $w\neq 0$ such that
\[
2 w^\top B^\top Bw + \lambda q(q-1) \sum_{i=1}^S \left( \|y^*_{\mathcal{G}_i}\|_1^{q-2} \cdot \left(\sum_{j\in \mathcal{G}_i} w_j{\rm sign}(y_j^*)\right)^2\right)=0.
\]
Let  $h:\R \to \R$ with $h(t):=\|B(y^*+tw)-b\|_2^2+\lambda\|y^*+tw\|_p^p$. Clearly, $h(\cdot)$ has a local minimum at 0, and $h(\cdot)$ is smooth around 0 with its derivatives being
\[
h'(0)=2 w^\top B^\top (By^*-b)+\lambda q \sum_{i=1}^S \left( \|y^*_{\mathcal{G}_i}\|_1^{q-1} \cdot \sum_{j\in \mathcal{G}_i} w_j{\rm sign}(y_j^*)\right)=0,
 \]
 \[
 h''(0)=2 w^\top B^\top Bw + \lambda q(q-1) \sum_{i=1}^S \left( \|y^*_{\mathcal{G}_i}\|_1^{q-2} \cdot \left(\sum_{j\in \mathcal{G}_i} w_j{\rm sign}(y_j^*)\right)^2\right) =0,
   \]
   \[
h^{(3)}(0)=\lambda q (q-1)(q-2) \sum_{i=1}^S \left( \|y^*_{\mathcal{G}_i}\|_1^{q-3} \cdot \left(\sum_{j\in \mathcal{G}_i} w_j{\rm sign}(y_j^*)\right)^3\right)=0,
 \]
   and
\begin{equation}\label{eq-SOG-2b}
h^{(4)}(0)=\lambda q (q-1)(q-2)(q-3) \sum_{i=1}^S \left( \|y^*_{\mathcal{G}_i}\|_1^{q-4} \cdot \left(\sum_{j\in \mathcal{G}_i} w_j{\rm sign}(y_j^*)\right)^4\right)<0.
\end{equation}
However, it is clear that $h^{(4)}(0)$ must be nonnegative, which yields a contradiction to \eqref{eq-SOG-2b}.  Therefore, we proved \eqref{eq-SOG-2}.
  \item By the structure of $y^*$ (cf. \eqref{eq-structure}), $\varphi(\cdot)$ is smooth around $y^*$, and thus, $(f+\varphi)(\cdot)$ is also smooth around $y^*$ with its derivatives being
   \[
   f'(y^*)+\varphi'(y^*)=0\quad {\rm and}\quad f''(y^*)+\varphi''(y^*)\succ 0
   \]
   (due to \eqref{eq-SOG-1} and \eqref{eq-SOG-2}).
Hence the second-order growth condition \eqref{eq-SOG-4} follows from \cite[Theorem 13.24]{Roc98}. This completes the proof.
\end{enumerate}
\end{proof}

The key of convergence rate analysis of PGM-GSO is the descent of the functional $f+\varphi$ in each iteration step.
The following lemma states some basic properties of active groups of sequence $\{x^k\}$ generated by the PGM-GSO.

\begin{lemma}\label{lem-LC}
Let $p=1$ and $0<q<1$. Let $\{x^k\}$ be a sequence generated by the {\rm PGM-GSO} with $v< \frac{1}{2}\|A\|_2^{-2}$, which converges to $x^*$ (by Theorem \ref{thm-GP-SO-2}). Let the assumptions and notations used in Lemma \ref{thm-SOG} be adopted.
We further define
\[
\alpha:=\|B\|_2^2,\quad L:=2\|A\|_2^2 \quad {\rm and} \quad D_k:=\varphi(y^k)-\varphi(y^{k+1})+ \langle f'(y^k), y^k-y^{k+1}\rangle.
\]
Then there exist some $\delta>0$ and $N\in \mathbb{N}$ such that the following inequalities hold for any $w\in B(y^*,\delta)$ and any $k\ge N$:
\begin{equation}\label{eq-lem-LC-1}
\varphi(w)-\varphi(y^{k+1})+ \langle f'(y^k), w-y^{k+1}\rangle \ge \frac1v \langle y^k- y^{k+1}, w-y^{k+1}\rangle-\alpha \|w-y^{k+1}\|_2^2,
\end{equation}
\begin{equation}\label{eq-lem-LC-2}
D_k \ge \left(\frac1v -\alpha\right) \| y^k- y^{k+1}\|_2^2,
\end{equation}
and
\begin{equation}\label{eq-lem-LC-3}
(f+\varphi)(y^{k+1})\le (f+\varphi)(y^{k})- \left(1-\frac{Lv}{2(1-v\alpha)}\right)D_k.
\end{equation}
\end{lemma}
\begin{proof}
By Lemma \ref{lem-LB}(ii) and the fact that $\{x^k\}$ converges to $x^*$, one has that $x^k$ shares the same index set of nonzero groups with that of $x^*$ for large $k$; further by the structure of $y^*$ (cf. \eqref{eq-structure}), we obtain that all components in nonzero groups of $y^k$ are nonzero for large $k$.
In another word, we have
\begin{equation}\label{eq-ActiveG}
\mbox{there exists $N\in \mathbb{N}$ such that $y_k\neq_\mathbf{a} 0$ and $z^k=0$ for any $k\ge N$;}
\end{equation}
hence $\varphi(\cdot)$ is smooth around $y^k$ for any $k\ge N$

In view of PGM-GSO and the decomposition of $x=\left({y}^\top,{z}^\top\right)^\top$, one has that
\begin{equation}\label{eq-LC-1}
y^{k+1}\in {\rm Arg}\min \left\{\lambda \varphi(y)+\frac{1}{2v}\left\|y-\left(y^k-v f'(y^k)\right)\right\|_2^2\right\}.
\end{equation}
The first-order necessary condition of \eqref{eq-LC-1} is
\begin{equation}\label{eq-LC-2}
\varphi'(y^{k+1})=\frac1v \left( y^k- v f'(y^k) -y^{k+1}\right).
\end{equation}
Recall from \eqref{eq-SOG-2} that $\varphi''(y^*)\succ -2B^\top B$. Since $\varphi(\cdot)$ is smooth around $y^*$, then there exists $\delta >0$ such that
$\varphi''(w)\succ -2B^\top B$ for any $w\in B(y^*,\delta)$.
Noting that $\{y^k\}$ converges to $y^*$, without loss of generality, we assume that $\|y^k-y^*\|<\delta$ for any $k\ge N$ (otherwise, we can choose a larger $N$). Therefore, one has that $\varphi''(y^k)\succ -2B^\top B$ for any $k\ge N$.
Then by Taylor expansion, we can assume without loss of generality that the following inequality holds for any $k\ge N$ and any $w\in B(y^*,\delta)$ (otherwise, we can choose a smaller $\delta$):
\begin{equation*}
\varphi(w)>\varphi(y^{k+1})+\langle \varphi'(y^{k+1}), w-y^{k+1}\rangle-\alpha \|w-y^{k+1}\|_2^2.
\end{equation*}
Hence, by \eqref{eq-LC-2}, it follows that
\begin{equation}\label{eq-LC-3}
\varphi(w)-\varphi(y^{k+1})>\frac1v \langle y^k- v f'(y^k) -y^{k+1}, w-y^{k+1}\rangle-\alpha \|w-y^{k+1}\|_2^2,
\end{equation}
which is reduced to \eqref{eq-lem-LC-1}, and \eqref{eq-lem-LC-2} follows by setting $w = y^k$ in \eqref{eq-lem-LC-1}.
Furthermore, by the definition of $f(\cdot)$, it is of class $C^{1,1}_L$ and it follows from \cite[Proposition A.24]{Bertsekas99} that
\[
\|f(y)-f(x)-f'(x)(y-x)\|\le \frac{L}{2}\|y-x\|^2\quad \mbox{for any } x,y.
\]
Then, by the definition of $D_k$, it follows that
\[
\begin{array}{lll}
(f+\varphi)(y^{k+1})-(f+\varphi)(y^{k})+D_k&=f(y^{k+1})-f(y^{k})+\langle f'(y^k), y^k-y^{k+1}\rangle\\
&\le \frac{L}{2}\|y^{k}-y^{k+1}\|_2^2\\
&\le \frac{Lv}{2(1-v\alpha)}D_k,
\end{array}
\]
where the last inequality follows from \eqref{eq-lem-LC-2}, and thus, \eqref{eq-lem-LC-3} is proved.
\end{proof}

The main result of this subsection is presented as follows, where we prove the linear convergence rate of the PGM-GSO to a local minimum for the case $p=1$ and $0<q<1$ under some mild assumptions.

\begin{theorem}\label{thm-LC}
Let $p=1$ and $0<q<1$.
Let $\{x^k\}$ be a sequence generated by the {\rm PGM-GSO} with $v< \frac{1}{2}\|A\|_2^{-2}$. Then $\{x^k\}$ converges to a critical point $x^*$ of \eqref{eq-GSOP}. Further assume that \ $x^*$ is a local  minimum  of \eqref{eq-GSOP}, and that any nonzero group of $x^*$ is active.
Then there exist $N\in \mathbb{N}$, $C>0$ and $\eta\in (0,1)$ such that
\begin{equation}\label{eq-thm-LC}
F(x^{k})-F(x^*)\le C\eta^k\quad {\rm and} \quad \|x^k-x^*\|_2\le C\eta^k,\quad \mbox{for any } k\ge N.
\end{equation}
\end{theorem}
\begin{proof}
The convergence of $\{x^k\}$ to a critical point $x^*$ of \eqref{eq-GSOP} directly follows from Theorem \ref{thm-GP-SO-2}.
Let notations used in Lemma \ref{thm-SOG} be adopted, $D_k$, $N$ and $\delta$ be defined as in Lemma \ref{lem-LC}, and let
\[r_k:=F(x^{k})-F(x^*).\]
%
Note in \eqref{eq-ActiveG} that $y^k\neq_\mathbf{a} 0$ and $z^k=0$ for any $k\ge N$. Thus
\[
r_k=(f+\varphi)(y^k)-(f+\varphi)(y^*)\quad \mbox{for any } k\ge N.
\]
It is trivial to see that $\varphi(\cdot)$ is smooth around $y^*$ (as it is active), and that
\[
\varphi''(y^*)=\lambda q(q-1)\left(\begin{array}{cccc}  M^*_1 &0 & 0\\ 0&\ddots &0 \\ 0&0 & M^*_S\\
\end{array}\right)\prec 0,\quad f''(y^*)+\varphi''(y^*)\succ 0
\]
(as shown in \eqref{eq-SOG-2}). This shows that $\varphi(\cdot)$ is concave around $y^*$, while  $(f+\varphi)(\cdot)$ is convex around $y^*$.
Without loss of generality, we assume that $\varphi(\cdot)$ is concave and $(f+\varphi)(\cdot)$ is convex in $B(y^*,\delta)$ and that
$y^k\in B(y^*,\delta)$ for any $k\ge N$ (since $\{y^k\}$ converges to $y^*$).

By the convexity of $(f+\varphi)(\cdot)$ in $B(y^*,\delta)$, it follows that for any $k\ge N$
\begin{equation}\label{eq-LC-b1}
\begin{array}{lll}
r_k&=(f+\varphi)(y^{k})-(f+\varphi)(y^*)\\
&\le \langle f'(y^k)+\varphi'(y^k), y^k-y^* \rangle\\
&= \langle f'(y^k)+\varphi'(y^k), y^k-y^{k+1} \rangle + \langle f'(y^k)+\varphi'(y^k), y^{k+1}-y^* \rangle\\
&= D_k - \varphi(y^k)+\varphi(y^{k+1})+ \langle \varphi'(y^k), y^k-y^{k+1}\rangle+ \langle f'(y^k)+\varphi'(y^k), y^{k+1}-y^* \rangle.
\end{array}
\end{equation}
Noting that $\varphi(\cdot)$ is concave in $B(y^*,\delta)$, it follows that
\[
\varphi(y^k)-\varphi(y^{k+1})\ge \langle \varphi'(y^k), y^k-y^{k+1}\rangle.
\]
Consequently, \eqref{eq-LC-b1} is reduced to
\begin{equation}\label{eq-LC-b2}
\begin{array}{lll}
r_k&\le D_k + \langle f'(y^k)+\varphi'(y^k), y^{k+1}-y^* \rangle\\
&= D_k + \langle \varphi'(y^k)-\varphi'(y^{k+1}), y^{k+1}-y^* \rangle + \langle f'(y^k)+\varphi'(y^{k+1}), y^{k+1}-y^* \rangle\\
&\le D_k + \left( \frac{L_{\varphi}}{2} + \frac1v\right)\|y^k-y^{k+1}\|_2 \|y^{k+1}-y^*\|_2,
\end{array}
\end{equation}
where the last inequality follows from the smoothness of $\varphi$ on $B(y^*,\delta)$ and \eqref{eq-LC-2}, and $L_{\varphi}$ is the Lipschitz constant of $\varphi'(\cdot)$ on $B(y^*,\delta)$. Let $\beta:=1-\frac{Lv}{2(1-v\alpha)}\in (0,1)$ (due to the assumption $v<\frac{1}{L}$). Then \eqref{eq-lem-LC-3} is reduced to
\[
r_k-r_{k+1}=(f+\varphi)(y^k)-(f+\varphi)(y^{k+1})\ge \beta D_k>0,
\]
and thus, it follows from \eqref{eq-LC-b2} and \eqref{eq-lem-LC-2} that
\begin{equation}\label{eq-LC-b3}
\begin{array}{lll}
\beta r_k &\le \beta D_k + \beta \left( \frac{L_{\varphi}}{2} + \frac1v\right)\|y^k-y^{k+1}\|_2 \|y^{k+1}-y^*\|_2\\
&\le r_k-r_{k+1} + \beta \left( \frac{L_{\varphi}}{2} + \frac1v\right) \|y^{k+1}-y^*\|_2 \sqrt{\frac{v}{1-v\alpha} D_k}\\
&\le r_k-r_{k+1} +  \left( \frac{L_{\varphi}}{2} + \frac1v\right) \sqrt{\frac{v\beta}{1-v\alpha}} \|y^{k+1}-y^*\|_2 \sqrt{r_k-r_{k+1}}.
\end{array}
\end{equation}
Recall from Lemma \ref{thm-SOG}(ii), there exists $c>0$ such that
\begin{equation*}
\|y-y^*\|_2^2\le c \left((f+\varphi)(y)- (f+\varphi)(y^*)\right) \quad\forall y\in B(y^*,\delta).
\end{equation*}
Thus, it follows that
\begin{equation}\label{eq-LC-b5}
\|y^{k+1}-y^*\|_2^2\le c r_{k+1} \le c r_k \quad \mbox{for each } k\ge N.
\end{equation}
Let $\epsilon:=\frac{c}{\beta}\left( \frac{L_{\varphi}}{2} + \frac1v\right)^2$.
By Young's inequality, \eqref{eq-LC-b3} yields that
\begin{equation}\label{eq-LC-b4}
\begin{array}{lll}
\beta r_k &\le r_k-r_{k+1} + \frac{1}{2\epsilon} \|y^{k+1}-y^*\|_2^2 \left( \frac{L_{\varphi}}{2} + \frac1v\right)^2 + \frac{\epsilon v\beta}{2(1-v\alpha)}(r_k-r_{k+1})\\
&\le r_k-r_{k+1} + \frac{\beta}{2} r_k + \frac{cv}{2(1-v\alpha)} \left( \frac{L_{\varphi}}{2} + \frac1v\right)^2(r_k-r_{k+1}).
\end{array}
\end{equation}
Let $\gamma:=\frac{cv}{2(1-v\alpha)} \left( \frac{L_{\varphi}}{2} + \frac1v\right)^2 > 0$. Then \eqref{eq-LC-b4} is reduced to
\[
r_{k+1}\le \frac{1+\gamma-\frac{\beta}{2}}{1+\gamma} r_k=\eta_1 r_{k},
\]
where $\eta_1:=\frac{1+\gamma-\frac{\beta}{2}}{1+\gamma} \in (0,1)$.
Thus, by letting $C_1:=r_N\eta_1^{-N}$, it follows that
\[
r_k\le \eta_1^{k-N} r_N = C_1\eta_1^k \quad  \mbox{for any } k\ge N.
\]
By letting $\eta_2=\sqrt{\eta_1}$ and $C_2=\sqrt{c C_1}$, it follows from \eqref{eq-LC-b5} that
\[
\|x^k-x^*\|_2=\|y^k-y^*\|_2\le (cr_k)^{1/2}\le  C_2\eta_2^k,\quad \mbox{for any } k\ge N.
\]
By letting $C:=\max\{C_1,C_2\}$ and $\eta:=\max\{\eta_1,\eta_2\}$, we arrive at \eqref{eq-thm-LC}.
The proof is complete.
\end{proof}

Theorem \ref{thm-LC} is an important theoretical result in that it establishes the linear convergence rate of proximal gradient method for solving the $\ell_{1,q}$ regularization problem under the assumption that any nonzero group of local minimum is an active group. Note that this assumption is satisfied automatically for the sparse optimization problem ($n_{max}=1$). Hence, when $n_{max}=1$, we obtain the linear convergence rate of proximal gradient method for solving $\ell_q$ regularization problem ($0<q<1$), which includes the iterative half thresholding algorithm ($q=1/2$) proposed in \cite{XuZB12} as a special case. This result is stated below, which we believe to the best of our knowledge that it is new.

\begin{corollary}
Let $0<q<1$, and let $\{x^k\}$ be a sequence generated by the proximal gradient method for solving the following $\ell_q$ regularization problem
\begin{equation}\label{eq-Lq}
\min_{x\in \R^n} F(x):=\|Ax-b\|_2^2+\lambda \|x\|_{q}^q
\end{equation}
with $v< \frac{1}{2}\|A\|_2^{-2}$. Then $\{x^k\}$ converges to a critical point $x^*$ of \eqref{eq-Lq}.
Further assume that $x^*$ is a local  minimum of \eqref{eq-Lq}.
Then there exist $N\in \mathbb{N}$, $C>0$ and $\eta\in (0,1)$ such that
\begin{equation}\label{eq-thm-LC}
F(x^{k})-F(x^*)\le C\eta^k\quad {\rm and} \quad \|x^k-x^*\|_2\le C\eta^k,\quad \mbox{for any } k\ge N.
\end{equation}
\end{corollary}

\section{Numerical experiments}
The purpose of this section is to carry out the numerical experiments of the proposed proximal gradient method for the $\ell_{p,q}$ regularization problem.
We illustrate the performance of the PGM-GSO among different types of $\ell_{p,q}$ regularization, in particular, when $(p,q)=(2,1),(2,0),(2,1/2),(1,1/2),(2,2/3)$ and $(1,2/3)$, and compare them with several state-of-the-art algorithms, for both simulated data and real data in gene transcriptional regulation.
All numerical experiments are implemented in MATLAB R2013b and executed on a
personal desktop (Intel Core Duo E8500, 3.16 GHz, 4.00 GB of RAM).
\subsection{Simulated data}
In the numerical experiments on simulated data, the numerical data are generated as follows.
We first randomly generate an i.i.d. Gaussian ensemble $A\in \R^{m\times n}$ satisfying $A^\top A=I$.
Then we generate a group sparse solution $\bar{x}\in \R^n$ via randomly splitting its components into $r$ groups
and randomly picking $k$ of them as active groups, whose entries are also randomly generated as i.i.d. Gaussian, while the remaining groups are all set as zeros. We generate the data $b$ by the MATLAB script
\[
b = A*\bar{x} + sigma*randn(m,1),
\]
where $sigma$ is the standard deviation of additive Gaussian noise.
The problem size is set to $n = 1024$ and $m = 256$, and we test on the noisy measurement data with $sigma=0.1\%$.
Assuming the group sparsity level $S$ is predefined, the regularization parameter $\lambda$ is iteratively updated by obeying the rule:
we set the iterative threshold to be the $S$-th largest value of $\|z^k_{\mathcal{G}_i}\|_2$ and solve the $\lambda$ by virtue of Theorem \ref{thm-Lp,q-formula}.

For each given sparsity level, which is $k/r$, we randomly generate the data $A$, $\bar{x}$, $b$ (as above) 500 times,
run the algorithm, and average the 500 numerical results to illustrate the performance of the algorithm.
We choose the stepsize $v=1/2$ in all the testing.
The two key criteria to characterize the performance are the relative error $\|x-\bar{x}\|_2/\|\bar{x}\|_2$ and the successful recovery rate, where the recovery is defined as \emph{success} when the relative error between the recovered data and the true data is smaller than $0.5\%$, otherwise, it is regarded as \emph{failure}.

We carry out six experiments with the initial point $x_0=0$ (unless otherwise specified).
In the first experiment, setting $r=128$ (so group size $G=1024/128=8$), we compare the convergence rate results and the successful recovery rates of the PGM-GSO with $(p,q)=(2,1),(2,0),(2,1/2),(1,1/2),(2,2/3)$ and $(1,2/3)$ for different sparsity levels. In Figure \ref{fig-ConvergenceRate}, (a), (b), and (c) illustrate the convergence rate results on sparsity level $1\%$, $5\%$, and $10\%$, respectively, while (d) plots the successful recovery rates on different sparsity levels.
When the solution is of high sparse level, as shown in Figure \ref{fig-ConvergenceRate}(a), all $\ell_{p,q}$ regularization problems perform perfect and achieve a fast convergence rate.
As demonstrated in Figure \ref{fig-ConvergenceRate}(b), when the sparsity level drops to $5\%$, $\ell_{p,1/2}$ and $\ell_{p,2/3}$ ($p=1$ and $2$) perform better and arrive at a more accurate level than $\ell_{2,1}$ and $\ell_{2,0}$.
As illustrated in Figure \ref{fig-ConvergenceRate}(c), when the sparsity level is $10\%$,
$\ell_{p,1/2}$ further outperforms $\ell_{p,2/3}$ ($p=1$ or $2$), and it surprises us that $\ell_{2,q}$  performs better and achieve a more accurate level than $\ell_{1,q}$ ($q=1/2$ or $2/3$).
From Figure \ref{fig-ConvergenceRate}(d), it is illustrated that $\ell_{p,1/2}$ achieves a better successful recovery rate than $\ell_{p,2/3}$ ($p=1$ or $2$), which outperforms $\ell_{2,0}$ and $\ell_{2,1}$. Moreover, we surprisingly see that $\ell_{2,q}$ also outperforms $\ell_{1,q}$ ($q=1/2$ or $2/3$) on the successful recovery rate. In a word, $\ell_{2,1/2}$ performs as the best one of these six regularizations on both accuracy and robustness.
In this experiment, we also note that the running times are at the same level, about 0.9 second per 500 iteration.

\begin{figure}[!htb]
\centering
  \includegraphics[width=16cm]{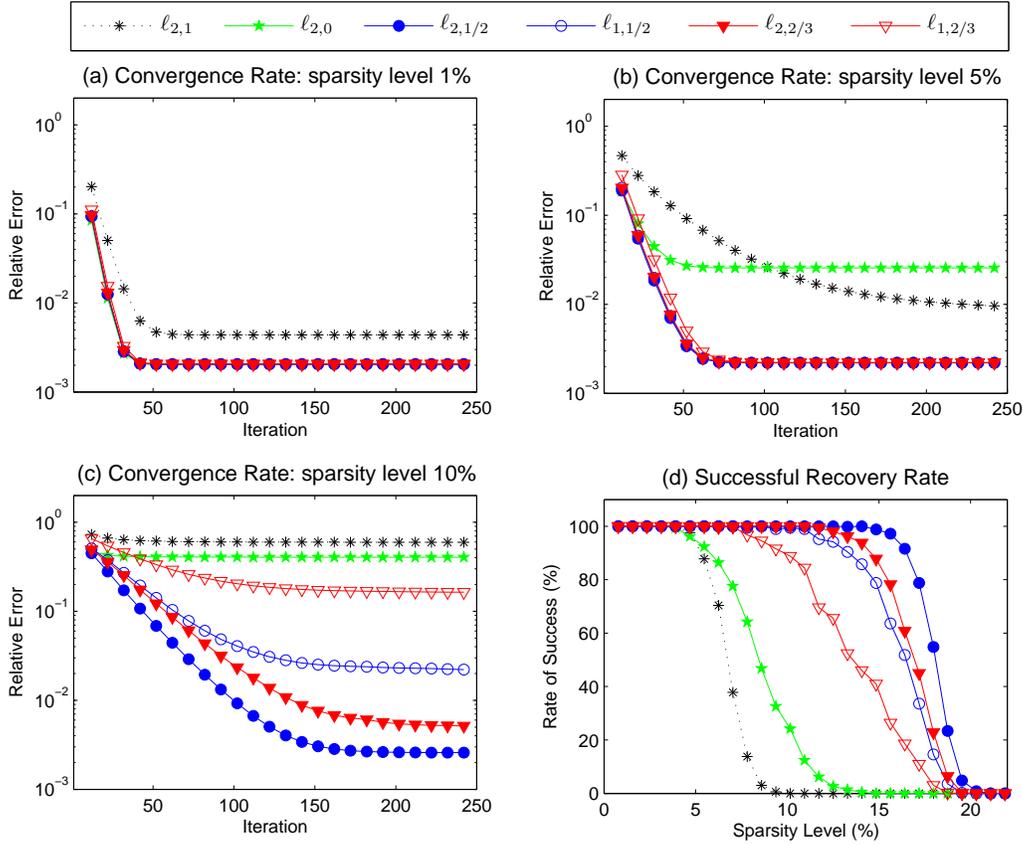}\\
  \caption{Convergence results and recovery rates for different sparsity levels.}
  \label{fig-ConvergenceRate}
\end{figure}

The second experiment is performed to show the sensitivity analysis on the group size ($G=4,8,16,32$) of the PGM-GSO
with the six types of $\ell_{p,q}$ regularization.
As shown in Figure \ref{fig-GroupingSensitivity}, the six types of $\ell_{p,q}$ reach a higher successful recovery rate for the larger group size.
We also note that the larger the group size, the shorter the running time.

\begin{figure}[!htb]
  \includegraphics[width=16cm]{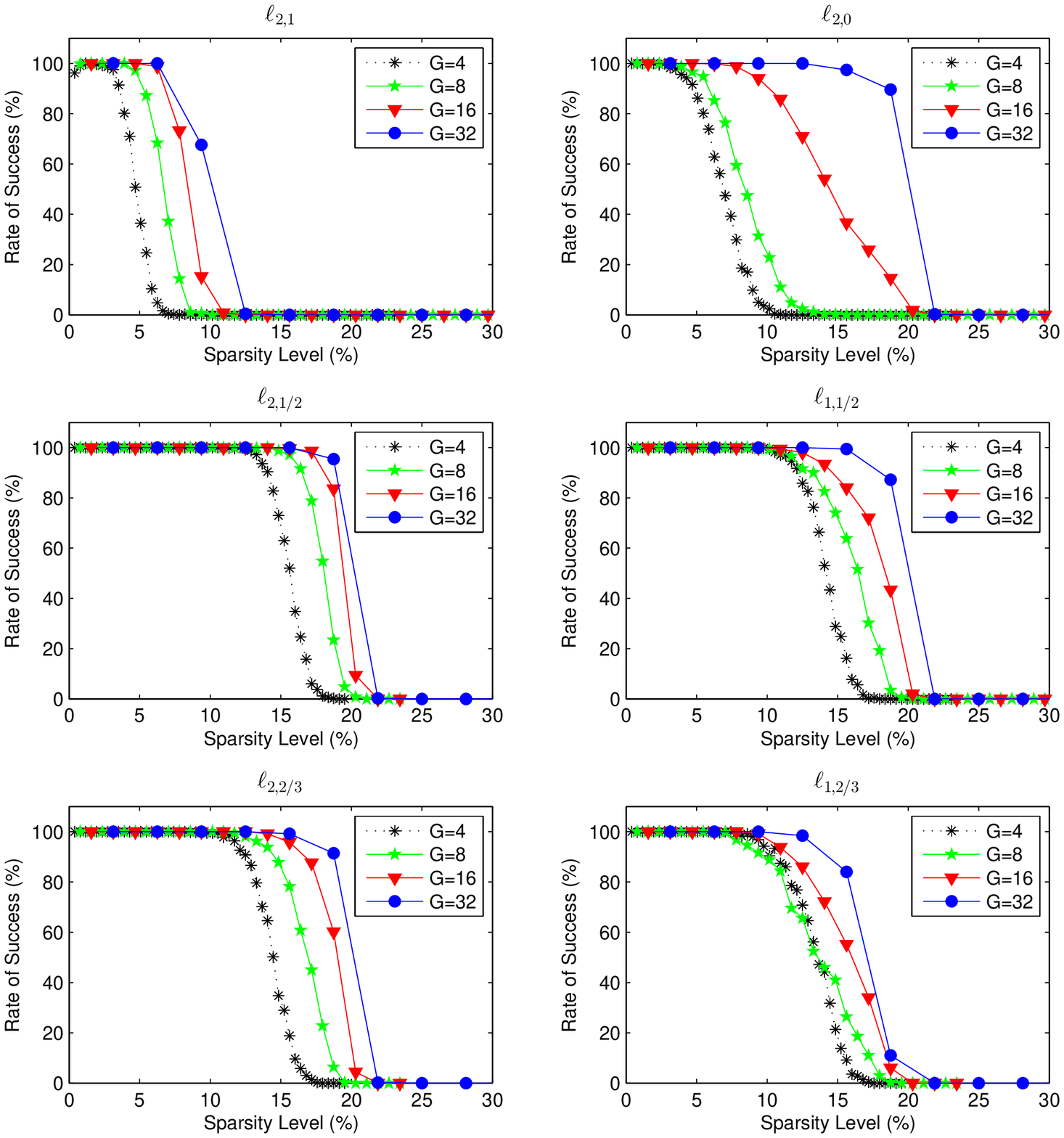}\\
  \caption{Sensitivity analysis on group size.}
  \label{fig-GroupingSensitivity}
\end{figure}

The third experiment is implemented to study the variation of the PGM-GSO when varying the regularization order $q$ (fix $p=2$). Recall from Theorem \ref{thm-Lp,q-formula}, the analytical solution of the proximal optimization subproblems \eqref{eq-PPA-Lp,q-subproblem} can be obtained when $q=0,1/2,2/3,1$. However, in other cases, the analytical solution of \eqref{eq-PPA-Lp,q-subproblem} seems not available, and thus we apply the Newton method to solve the nonlinear equation \eqref{eq-FOC}, which is the optimality condition of the proximal optimization subproblem. Figure \ref{fig-vary-q} shows the variation of successful recovery rates by decreasing the regularization order $q$ from 1 to 0. It is illustrated that the PGM-GSO achieves the best successful recovery rate when $q=1/2$, which arrives at the same
conclusion as the first experiment. The farther the distance of $q$ (in $[0,1]$) from $1/2$, the lower the successful recovery rate.

\begin{figure}[!htb]
  \center
  \includegraphics[width=12cm]{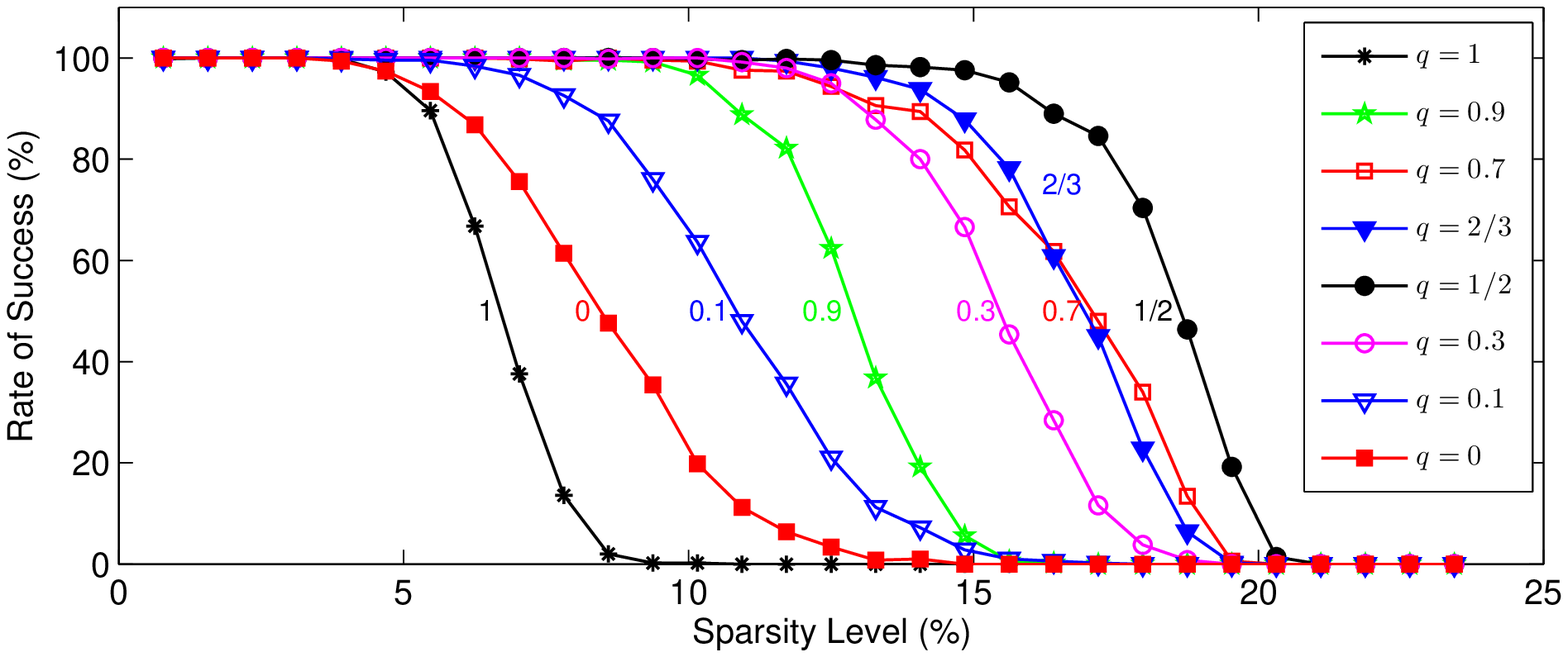}\\
  \caption{Variation of the PGM-GSO when varying the regularization order $q$.}
  \label{fig-vary-q}
\end{figure}

%
%

The fourth experiment is to compare the PGM-GSO with several state-of-the-art algorithms in the field of sparse optimization, either convex or nonconvex algorithms, including $\ell_1$-Magic\footnote{$\ell_1$-Magic is a collection of MATLAB routines for solving the convex optimization programs central to compressive sampling, based on standard interior-point methods. The package is available at http://users.ece.gatech.edu/\~{}justin/l1magic/.} \cite{CandesTao06},
YALL1\footnote{YALL1 (Your ALgorithm for L1) is a package of MATLAB solvers for the $\ell_1$ sparse reconstruction, by virtue of the alternating direction method. The package is available at http://yall1.blogs.rice.edu/.} \cite{YZhang11,YZ11},
GBM\footnote{GBM is a Gradient Based Method for solving the $\ell_{1/2}$ regularization problem. This algorithm is sensitive to the initial guess.
Suggested by the authors, we choose the initial point as the solution obtained by the $\ell_1$-Magic.} \cite{DHLiPJO14},
LqRecovery\footnote{LqRecovery is an iterative algorithm for the $\ell_p$ norm minimization. The code is available at http://www.math.drexel.edu/\~{}foucart/software.htm.}  \cite{FoucartLai09},
HardTA\footnote{HardTA is the iterative Hard Thresholding Algorithm, which is to solve the $\ell_{0}$ regularization problem.} \cite{Blumensath08,BrediesLorenz08}
and HalfTA\footnote{HalfTA is the iterative Half Thresholding Algorithm, which is to solve the $\ell_{1/2}$ regularization problem.} \cite{XuZB12}.
Figure \ref{fig-SAs} demonstrates the successful recovery rates of these algorithms on different sparsity levels.
It is indicated by Figure \ref{fig-SAs} that $\ell_{2,1/2}$ and $\ell_{2,2/3}$ can achieve the higher successful recovery rate
than other algorithms, by exploiting the group sparsity structure.

\begin{figure}[!htb]
  \center
  \includegraphics[width=12cm]{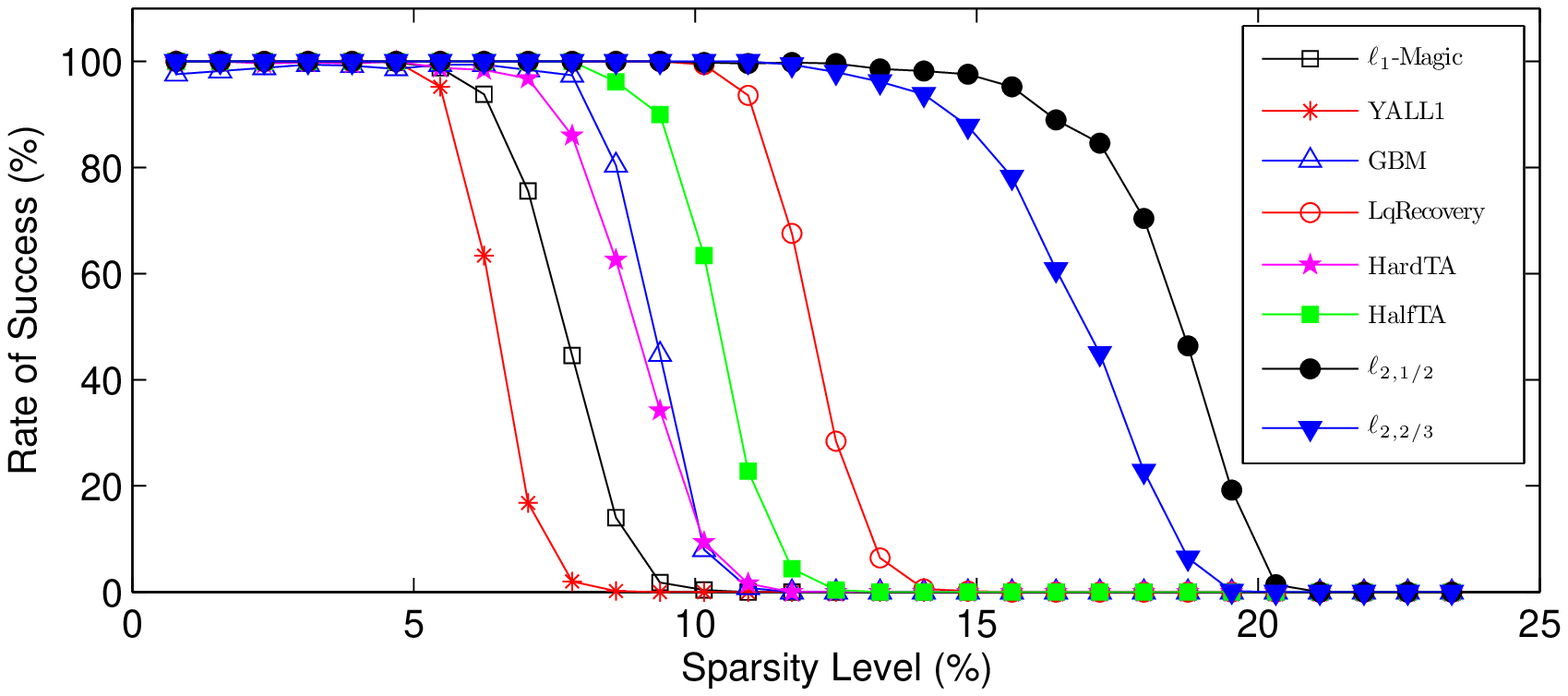}\\
  \caption{Comparison between the PGM-GSO and several state-of-the-art algorithms.}
  \label{fig-SAs}
\end{figure}

Even though some global optimization method, such as the filled function method \cite{Ge90},
can find the global solution of the lower-order regularization problem as in Example \ref{ex-REC-NE},
however, it does not work for the large-scale sparse optimization problems.
Because, in the filled function method, all the directions need to be searched or compared in each iteration, which costs a large amount of time and hampers the efficiency for solving the large-scale problems.

\subsection{Real data in gene transcriptional regulation}
Gene transcriptional regulation is the process that a combination of transcription factors (TFs) act in concert to control the transcription of the target genes. Inferring gene regulatory network from high-throughput genome-wide data is still a major challenge in systems biology, especially when the number of genes is large but the number of experimental samples is small. In large genomes, such as human and mouse, the complexity of gene regulatory system dramatically increases. Thousands of TFs combine in different ways to regulate tens of thousands target genes in various tissues or biological processes. However, only a few TFs collaborate and usually form complexes (i.e., groups of cooperative TFs) to control the expression of a specific gene in a specific cell type or developmental stage. Thus, the prevalence of TF complex makes the solution of gene regulatory network have a group structure, and the gene regulatory network inference in such large genomes becomes a group sparse optimization problem, which is to search a small number of TF complexes (or TFs) from a pool of thousands of TF complexes (or TFs) for each target gene based on the dependencies between the expression of TF complexes (or TFs) and the targets. Even though TFs often work in the form of complexes \cite{XieDan13Cell}, and TF complexes are very important in the control of cell identity and diseases \cite{Hnisz13Cell}, current methods to infer gene regulatory network usually consider each TF separately. To take the grouping information of TF complexes into consideration, we can apply the group sparse optimization to gene regulatory network inference with the prior knowledge of TF complexes as the pre-defined grouping.

\subsubsection{Materials}
Chromatin immunoprecipitation (ChIP) coupled with next generation sequencing (ChIP-seq) identifies \emph{in vivo} active and cell-specific binding sites of a TF. They are commonly used to infer TF complexes recently. Thus, we manually collect ChIP-seq data in mouse embryonic stem cells (mESCs) (Table 1). Transcriptome is the gene expression profile of the whole genome that is measured by microarray or RNA-seq. The transcriptome data in mESCs for gene regulatory network inference are downloaded from Gene Expression Omnibus (GEO). 245 experiments under perturbations in mESC are collected from three papers \cite{Correa-Cerro11,Nishiyama13,Nishiyama09}. Each experiment produced transcriptome data with or without overexpression or knockdown of a gene, in which the control and treatment have two replicates respectively. Gene expression fold changes between control samples and treatment samples of 12488 target genes in all experiments are $\log2$ transformed and form matrix $B\in \R^{245\times 12488}$ (Figure \ref{fig-Workflow}A). The known TFs are collected from four TF databases, TRANSFAC, JASPAR, UniPROBE and TFCat, as well as literature.
Let matrix $H\in \R^{245\times 939}$ be made up of the expression profiles of 939 known TFs, and matrix $Z\in \R^{939\times 12488}$ describe the connections between these TFs and targets. Then, the regulatory relationship between TFs and targets can be represented approximately by a linear system
\[
HZ=B+\epsilon.
\]
The TF-target connections defined by ChIP-seq data are converted into an initial matrix $Z^0$ (cf. \cite{QinHu2014}). Indeed, if TF $i$ has a binding site around the gene $j$ promoter within a defined distance (10 kbp), a non-zero number is assigned on $Z_{ij}^0$ as a prior value.


Now we add the grouping information (TF complexes) into this linear system. The TF complexes are inferred from ChIP-seq data (Table 1) via the method described in \cite{Gianno13GenRes}. Let the group structure of $Z$ be a matrix $W\in \R^{2257\times 939}$ (actually, the number of groups is 1439), whose Moore-Penrose pseudoinverse \cite{Horn85} is denoted by $W^+$. We further let $A:=HW^+$ and $X:=WZ$. Then the linear system can be converted into
\[
AX=B+\epsilon,
\]
where $A$ denotes expression profiles of TF complexes, and $X$ represents connections between TF complexes and targets (Figure \ref{fig-Workflow}A).

\begin{table}[!htb]
  \centering
  \caption{ChIP-seq data for TF complex inference.}
  {\scriptsize
\begin{tabular}{|ccc|ccc|}
\hline
\textbf{Factor} & \textbf{GEO accession} & \textbf{Pubmed ID} & \textbf{Factor} & \textbf{GEO accession} & \textbf{Pubmed ID} \\ \hline
Atf7ip	&GSE26680	&-	&Rad21	&GSE24029	&21589869\\
Atrx	&GSE22162	&21029860	&Rbbp5	&GSE22934	&21477851\\
Cdx2	&GSE16375	&19796622	&Rcor1	&GSE27844	&22297846\\
Chd4	&GSE27844	&22297846	&Rest	&GSE26680	&-\\
Ctcf	&GSE11431	&18555785	&Rest	&GSE27844	&22297846\\
Ctcf	&GSE28247	&21685913	&Rnf2	&GSE13084	&18974828\\
Ctr9	&GSE20530	&20434984	&Rnf2	&GSE26680	&-\\
Dpy30	&GSE26136	&21335234	&Rnf2	&GSE34518	&22305566\\
E2f1	&GSE11431	&18555785	&Setdb1	&GSE17642	&19884257\\
Ep300	&GSE11431	&18555785	&Smad1	&GSE11431	&18555785\\
Ep300	&GSE28247	&21685913	&Smad2	&GSE23581	&21731500\\
Esrrb	&GSE11431	&18555785	&Smarca4	&GSE14344	&19279218\\
Ezh2	&GSE13084	&18974828	&Smc1a	&GSE22562	&20720539\\
Ezh2	&GSE18776	&20064375	&Smc3	&GSE22562	&20720539\\
Jarid2	&GSE18776	&20064375	&Sox2	&GSE11431	&18555785\\
Jarid2	&GSE19365	&20075857	&Stat3	&GSE11431	&18555785\\
Kdm1a	&GSE27844	&22297846	&Supt5h	&GSE20530	&20434984\\
Kdm5a	&GSE18776	&20064375	&Suz12	&GSE11431	&18555785\\
Klf4	&GSE11431	&18555785	&Suz12	&GSE13084	&18974828\\
Lmnb1	&GSE28247	&21685913	&Suz12	&GSE18776	&20064375\\
Med1	&GSE22562	&20720539	&Suz12	&GSE19365	&20075857\\
Med12	&GSE22562	&20720539	&Taf1	&GSE30959	&21884934\\
Myc		&GSE11431	&18555785	&Taf3	&GSE30959	&21884934\\
Mycn	&GSE11431	&18555785	&Tbp	&GSE30959	&21884934\\
Nanog	&GSE11431	&18555785	&Tbx3	&GSE19219	&20139965\\
Nipbl	&GSE22562	&20720539	&Tcfcp2l1	&GSE11431	&18555785\\
Nr5a2	&GSE19019	&20096661	&Tet1	&GSE26832	&21451524\\
Pou5f1	&GSE11431	&18555785	&Wdr5	&GSE22934	&21477851\\
Pou5f1	&GSE22934	&21477851	&Whsc2	&GSE20530	&20434984\\
Prdm14	&GSE25409	&21183938	&Zfx	&GSE11431	&18555785\\ \hline
\end{tabular} }
\end{table}

A literature-based golden standard (low-throughput golden standard) TF-target pair set from biological studies (Figure \ref{fig-Workflow}C),
including 97 TF-target interactions between 23 TFs and 48 target genes, is downloaded from iScMiD (Integrated Stem Cell Molecular Interactions Database). Each TF-target pair in this golden standard dataset has been verified by biological experiments.
Another more comprehensive golden standard mESC network is constructed from high-throughput data (high-throughput golden standard) by ChIP-Array \cite{Qin11NAR} using the methods described in \cite{QinHu2014}. It contains 40006 TF-target pairs between 13092 TFs or targets (Figure \ref{fig-Workflow}C). Basically, each TF-target pair in the network is evidenced by a cell-type specific binding site of the TF on the target's promoter and the expression change of the target in the perturbation experiment of the TF, which is generally accepted as a true TF-target regulation. These two independent golden standards are both used to validate the accuracy of the inferred gene regulatory networks.

\subsubsection{Numerical results}
%
We apply and compare the PGM-GSO, starting from the initial matrix $X^0:=WZ^0$, to the gene regulatory network inference problem (Figure \ref{fig-Workflow}B). The area under the curve (AUC) of a receiver operating characteristic (ROC) curve is widely recognized as an important index of the overall classification performance of an algorithm; see \cite{Fawcett06ROC}. Here, we apply AUC to evaluate the performance of the PGM-GSO with four types of $\ell_{p,q}$ regularization, $(p,q)=(2,1),(2,0),(2,1/2)$ and $(1,1/2)$. A series of numbers of predictive TF complexes (or TFs), denoted by $k$, from 1 to 100 (that is, the sparsity level varies from about $0.07\%$ to $7\%$) are tested. For each $k$ and each pair of TF complex (or TF) $i$ and target $j$, if the $X^{(k)}_{G_ij}$ is non-zero, this TF complex (or TF) is regarded as a potential regulator of this target in this test. In biological sense, we only concern about whether the true TF is predicted, but not the weight of this TF. We also expect that the TF complexes (or TFs) which are predicted in a higher sparsity level should be more important than those that are only reported in a lower sparsity level. Thus, when calculating the AUC, a score ${\rm Score}_{ij}$ is applied as the predictor for TF $i$ on target $j$:
\begin{equation*}
 {\rm Score}_{ij}:=\left\{\begin{matrix}
   \max_k\{1/k\},&{X_{i j}^{(k)}\neq 0,}\\
   0, &{\rm otherwise.}
\end{matrix}\right.
\end{equation*}

\begin{figure}[!htb]
\centering
\mbox{ \subfigure[Evaluation with high-throughput golden standard.]{\includegraphics[width=7.5cm]{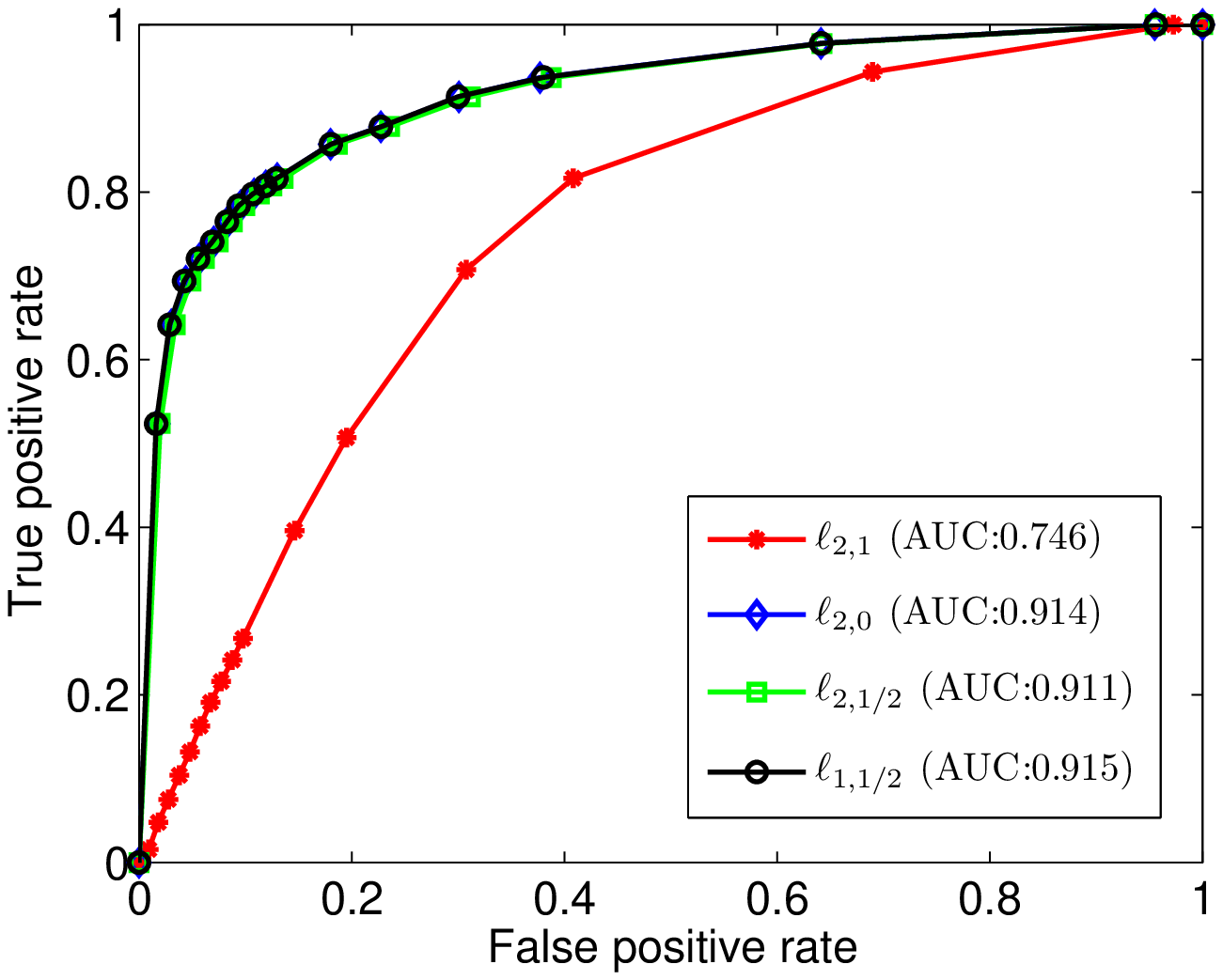}} \quad
\subfigure[Evaluation with literature-based low-throughput golden standard.]{\includegraphics[width=7.5cm]{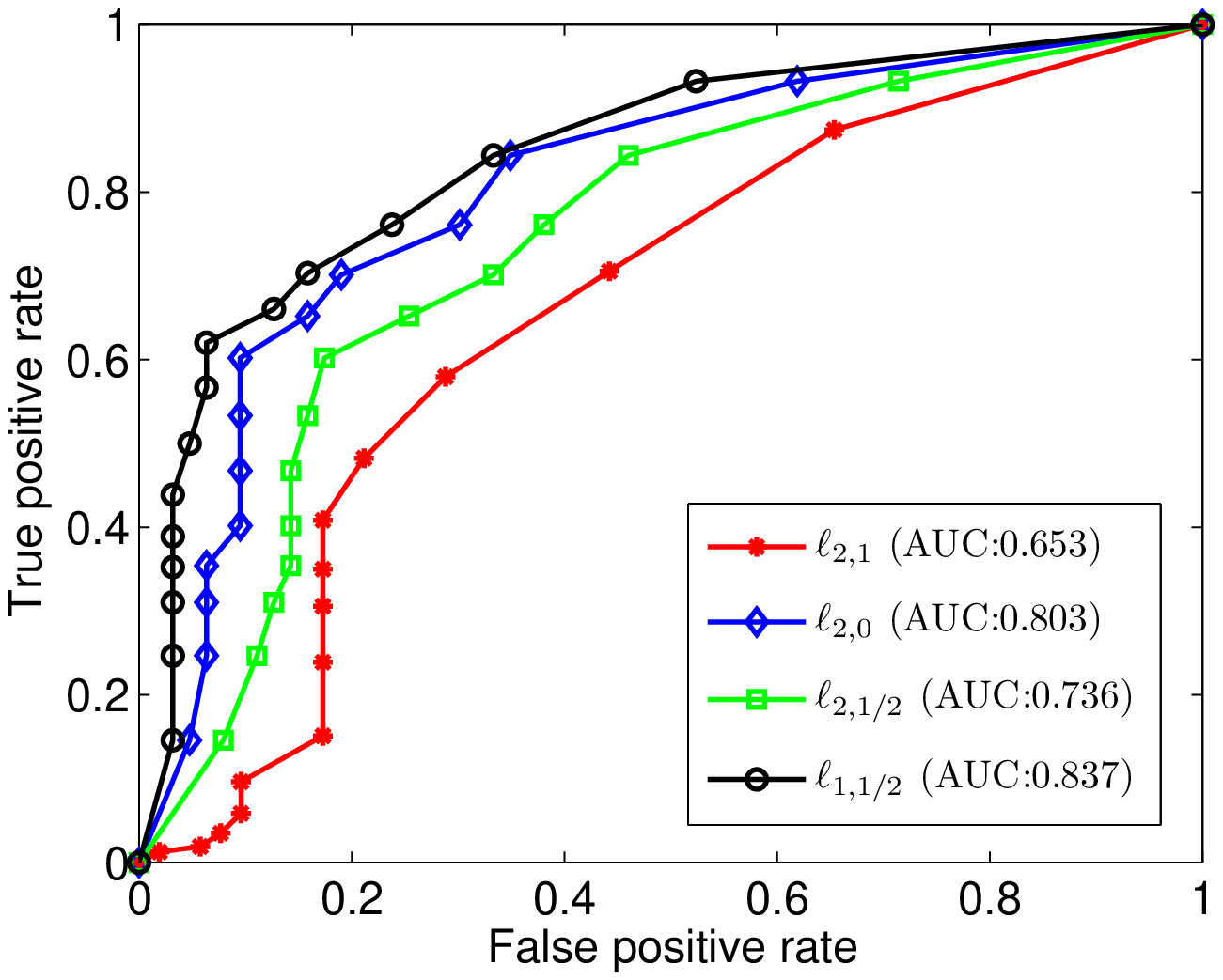}} \quad }
\caption{ROC curves and AUCs of the PGM-GSO on mESC gene regulatory network inference.}
\label{fig-ROC}
\end{figure}

Both high-throughput and low-throughput golden standards are used to draw the ROC curves of the PGM-GSO with four types of $\ell_{p,q}$ regularization in Figure \ref{fig-ROC} to compare their accuracy. When matched with the high-throughput golden standard, it is illustrated from Figure \ref{fig-ROC}(a) that $\ell_{2,1/2}$, $\ell_{1,1/2}$ and $\ell_{2,0}$ perform almost the same (as indicated by the almost same AUC value), and significantly outperform $\ell_{2,1}$.
With the low-throughput golden standard, it is demonstrated from Figure \ref{fig-ROC}(b) that $\ell_{1,1/2}$ is slightly better than $\ell_{2,1/2}$ and $\ell_{2,0}$, and these three regularizations perform much better than $\ell_{2,1}$.
These results are basically consistent with the results from simulated data. Since the golden standards we use here are obtained from real biological experiments, which are well-accepted as true TF-target regulations, the higher AUC, the more biologically accurate the result gene regulatory network is. Thus, our results indicate that the $\ell_{p,1/2}$ and $\ell_{p,0}$ regularizations are applicable to gene regulatory network inference in biological researches that study higher organisms but generate transcriptome data for only a small number of samples, which facilitates biologists to analyze gene regulation in a system level.

\noindent \vskip 0.5cm

\noindent \textbf{Acknowledgment.}
We are grateful to the four anonymous reviewers for their valuable suggestions and remarks which helped to improve the quality of the paper.
We are also thankful to Professor Marc Teboulle for providing the reference \cite{BolteTeboulle13} and the suggestion that the global convergence of the proximal gradient method for the nonconvex and nonsmooth composite optimization problem can be established by using the so-called Kurdyka-{\L}ojasiewicz theory. Indeed, we show in Theorem \ref{thm-GP-SO-2} that the $\ell_{p,0}$ regularization problem globally converges to a local minimum and the $\ell_{p,q}$ $(0<q<1)$ regularization problem globally converges to a critical point by virtue of the Kurdyka-{\L}ojasiewicz theory.



\end{document}